\documentclass[reqno,12pt]{amsart}
\usepackage{a4wide}

\usepackage{varioref}   %
\usepackage{hyperref}
\usepackage{cleveref}

\usepackage{amsmath,amssymb,amsthm,amsfonts}
\usepackage{mathrsfs}
\usepackage{bbm}
\usepackage{hyperref}

\usepackage{color}

\newtheorem{lemma}{Lemma}[section]
\newtheorem{theorem}{Theorem}[section]

\newtheorem{proposition}{Proposition}[section]

\numberwithin{equation}{section}

\newcommand{\dis}{\displaystyle}

\newcommand{\CE}{\mathcal{E}}

\newcommand{\CH}{\mathcal{H}}

\newcommand{\CL}{\mathcal{L}}
\newcommand{\CN}{\mathcal{N}}

\newcommand{\CD}{\mathcal{D}}
\newcommand{\CR}{\mathcal{R}}

\newcommand{\la}{\lambda}

\newcommand{\si}{\sigma}

\newcommand{\Ga}{\Gamma}

\begin{document}

\title[plasma flow interacting with boundary]{The stability  of sheath to the nonisentropic Euler-Poisson system with fluid-boundary interaction}

\author[H.-Y. Yin]{Haiyan Yin}
\address[HYY]{School of Mathematical Sciences, Huaqiao University, Quanzhou 362021, P.R.~China.}
\email{hyyin@hqu.edu.cn}

\author[R. Zeng]{Rong Zeng}
\address[RZ]{School of Mathematical Sciences, Huaqiao University, Quanzhou 362021, P.R.~China.}
\email{18770610525@163.com}

\author[M.-m. Zhu]{Mengmeng Zhu}
\address[MMZ]{School of Mathematics, Northwest University, Xi'an 710127, P.R.~China.}
\email{zhumm0907@163.com}

\keywords{nonisentropic Euler-Poisson system, sheath, fluid-boundary interaction, asymptotic stability, convergence rate, weighted energy method.\\
\mbox{\quad 2020 Mathematics Subject Classification: 76X05, 35M13, 35B35, 35B40, 35B45.}\\
\mbox{\quad Corresponding author: Mengmeng Zhu}.}
\date{\today}
\maketitle

\begin{abstract}
In the present paper, we define the sheath by a monotone stationary solution to the nonisentropic
Euler-Poisson system under a condition known as the Bohm criterion and consider a situation in
which charged particles accumulate on the boundary due to the flux from the inner region. Under
this fluid-boundary interactive setting, we prove the large time asymptotic stability of the sheath provided that the initial perturbation is sufficiently small in some weighted Sobolev spaces. Moreover, the convergence rate of the solution toward the sheath is obtained. The proof is based on the weighted energy method.
\end{abstract}

\section{Introduction}
 Mathematically, the plasma sheath is often described as the stationary solution or boundary layer solution in the half line to the Euler-Poisson system for the only heavier ions flow under the Boltzmann relation.
In the paper, the flow of positively charged ions in plasmas is governed by the nonisentropic Euler-Poisson system of the form
\begin{eqnarray}\label{1.1}
&&\left\{\begin{aligned}
& n_t+(nu)_{x}=0,\\
& ( mn u)_t+(mnu^{2}+p)_x
=n\phi_x,\\
&W_{t}+(Wu+pu)_{x}=nu\phi_x,\\
&\phi_{xx}=n-e^{-\phi}.
\end{aligned}\right.
\end{eqnarray}
The unknown functions $n$, $u$ and  $\phi$  stand for the density, velocity and the electrostatic potential, respectively.  The function $W$ stands for the total
energy given by
\begin{eqnarray}\label{1.2}
W=\frac{1}{2}mnu^{2}+\frac{p}{\gamma-1},
\end{eqnarray}
where the constant $\gamma>1$ is the ratio of specific heats and the pressure $p$ satisfies the equation of state:
\begin{eqnarray}\label{1.3}
p=RTn,
\end{eqnarray}
with the temperature function $T$ and the Boltzmann constant $R>0$.
Note that $\phi$ has been chosen to have an opposite sign compared to the usual situation in physics. In the fourth equation of \eqref{1.1}, the electron density $n_{e}$ is determined by the  electrostatic potential in terms of the Boltzmann relation $n_{e}=e^{-\phi}$. We remark that the nonisentropic Euler-Poisson system \eqref{1.1} can be formally derived through the macro-micro decomposition from the Vlasov-Poisson-Boltzmann system for the ions flow in kinetic theory, cf.~\cite{DL}.
 Substituting \eqref{1.2} and \eqref{1.3} into the equations
 $\eqref{1.1}$, we can obtain the following equations:
\begin{eqnarray}\label{1.4}
&&\left\{\begin{aligned}
& n_t+(nu)_{x}=0,\\
&  mn (u_t+uu_{x})+(RTn)_x
=n\phi_x,\\
&T_{t}+uT_{x}+(\gamma-1)Tu_{x}=0,\\
&\phi_{xx}=n-e^{-\phi}.
\end{aligned}\right.
\end{eqnarray}

The goal of this paper is to study the large-time behavior of solutions to the initial boundary value problem on \eqref{1.4} over the one-dimensional half space $\mathbb{R}_{+}:=\{x>0\}$.

Initial data for system \eqref{1.4} are given by
\begin{eqnarray}\label{1.5}
(n,u,T)(0,x)=(n_{0},u_{0},T_{0})(x)\quad \text{with}\ \
\inf_{x\in \mathbb{R}_{+}}n_{0}(x)>0\ \ \text{and} \ \ \inf_{x\in \mathbb{R}_{+}}T_{0}(x)>0,
\end{eqnarray}
\begin{eqnarray}\label{1.5a}
\lim_{x\rightarrow\infty}(n_{0},u_{0},T_{0})(x)=(n_{\infty},u_{\infty},T_{\infty}),
\end{eqnarray}
\begin{eqnarray}\label{1.jh5a}
\phi_{x}(0,0)=q_{0}\in \mathbb{R},
\end{eqnarray}
where $n_{\infty}>0$, $u_{\infty}<0$, $T_{\infty}>0$ and $q_{0}$ are constants.

The boundary data for system \eqref{1.4} are given by
\begin{eqnarray}\label{1.5b}
\phi_{xt}(t,0)=[-nu-u_{e}e^{-\phi}](t,0)
\end{eqnarray}
and
\begin{eqnarray}\label{1.5bgh}
 \lim_{x\rightarrow\infty}\phi(t,x)=0,
\end{eqnarray}
where the thermal velocity of electrons $u_{e}>0$ is constant.

In the present paper, we deal
with a situation where charged particles accumulate on the surface. For this purpose,
we adopt the boundary condition \eqref{1.5b} from \cite{Cip}. Since $\phi_{x}(t,0)$ is positively proportional to the net quantity of charged particles accumulated on the boundary, its temporal derivative is equal to the net flux of charged particles $[-nu-n_{e}u_{e}](t,0)$
as \eqref{1.5b} claims due to $n_{e}=e^{-\phi}$. Moreover, we always assume that
\begin{eqnarray}\label{1.6}
n_{\infty}=1,
\end{eqnarray}
so that the quasi-neutrality  holds true at $x=\infty$ by \eqref{1.5bgh} and \eqref{1.6}. Motivated by
\cite{boundaawa} which studied the asymptotic stability of sheaths for the isentropic Euler-Poisson system with the fluid-boundary interaction condition,
it would be also worth for us extending  this result to  the nonisentropic case.

Nowadays there are many mathematical studies on the sheath formation by using the Euler-Poisson system. In what follows, we will divide the flows into three cases: isothermal, isentropic and nonisentropic to state and review.
Physicists study the sheath formation mainly for the case of isothermal flows(i.e., $p(\rho)=K\rho)$.
For the isothermal flows(i.e., $p(\rho)=K\rho)$, Suzuki \cite{Suzuki} first showed that the Bohm criterion gave a sufficient condition for an existence of the stationary solution by the phase plane method for Euler-Poisson system. It was shown that the stationary solution was time asymptotically stable provided that an initial perturbation was sufficiently small in the weighted Sobolev space and the convergence rate of the time global solution towards the stationary solution was obtained.
Specifically, Suzuki \cite{multicomponent} further treated the same problem for a multi-component plasma.
 Subsequently, Nishibata-Ohnawa-Suzuki \cite{NOS} showed that the stationary solutions to the Dirichlet
problem over one-, two-, and three-dimensional half space were asymptotically stable
under the Bohm criterion by making use of the weighted energy method. Recently, Suzuki-Takayama  \cite{Suzuki04} studied the existence and asymptotic
stability of stationary solutions for Euler-Poisson system in a three-dimensional space domain of which
boundary was drawn by a graph.

For the isentropic flows (i.e., $p(\rho)=K\rho^{\gamma}, \gamma>1$), Ambroso-M\'ehats-Raviart \cite{Asympto} studied the existence of monotone solutions for the stationary problem over a finite interval by solving the Poisson equation with the small Debye length via the singular perturbation approach. Later, Ambroso \cite{Ambros} gave a further study to determine the stationary solutions in terms of different levels of an associated energy functional and numerically showed which solution was asymptotically stable in large time. Ohnawa \cite{boundaawa} studied the existence and asymptotic stability of boundary layers for the fluid-boundary interaction condition that the time change rate of the electric field $-\phi_x(t,0)$ at the boundary was equal to the total flux of charges.
Recently, Chen-Ding-Gao-Lin-Ruan \cite{CheDR} studied
 the initial-boundary value problem on the Euler-Poisson system arising in plasma physics over one-, two-, and three-dimensional half space. By assuming that the velocity of the positive ion satisfies the Bohm criterion at the far field, they established the global unique existence and the large time asymptotic stability
of boundary layer (i.e., stationary solution) in some weighted Sobolev spaces by weighted energy method.

As mentioned above, the sheath formation for the isothermal or isentropic flow has been
well studied. There are few results about nonisentropic case. Duan-Yin-Zhu \cite{SIAm} gave the first result which investigated the sheath
formation about the nonisentropic Euler-Poisson system. They showed the unique existence and
asymptotic stability of the monotone stationary solutions over a half line under the Bohm criterion saying that ions must move toward the wall at infinity with a velocity greater than a critical value given particularly as the acoustic velocity for cold ions.
Yao-Yin-Zhu \cite{yaoyin} extended the results in \cite{SIAm} to $N$-dimensional ($N$=1,2,3) half space.
By assuming that the velocity of the positive ion satisfied the Bohm criterion at
the far field, they established the global unique existence and the large time asymptotic stability
of the sheath in some weighted Sobolev spaces by weighted
energy method. A key different point from \cite{SIAm} was to derive some boundary estimates on the derivative of the potential in the $x_1$-direction.
Li-Suzuki \cite{LiMing} mathematically investigated the formation of a
plasma sheath near the surface of nonplanar walls. They studied the existence and asymptotic
stability of stationary solutions for the nonisentropic Euler-Poisson system in a domain
of which boundary was drawn by a graph, by employing a space weighted energy method.

Of course, it is also intriguing to study the sheath formation by using kinetic models such as the Vlasov-Poisson system.
The papers \cite{Suzuki03, Suzuki02} studied the unique existence and nonlinear stability of the stationary solution in a half line under the kinetic Bohm criterion.
For more details of physicality of the sheath development, we
refer the reader to \cite{FFC,Riemann}. Moreover, we also mention \cite{Rousset, GHR, JKS, Jung} for the problem on the quasineutral limit of Euler-Poisson system of the ions flow in the presence of boundaries and \cite{GGPS} for the derivation of the ions equations from the general two-fluid model in plasma physics. Meanwhile, for the Cauchy problem on Euler-Poisson system of the similar form \eqref{1.1} for ions, we may refer to \cite{CG,GPu,LLS,LS,LY,Pu} and references therein for the extensive studies of the dispersive property.

In this paper, we are only concerned with the existence and asymptotic stability of stationary solutions to the initial boundary value problem \eqref{1.4}-\eqref{1.6}. For this purpose, we denote $(\tilde n, \tilde
u, \tilde T, \tilde\phi)$ to be the solution of the stationary problem on the half space
\begin{eqnarray}\label{1.7}
&&\left\{\begin{aligned}
&(\tilde n\tilde u)_{x}=0,\\
& m\tilde n \tilde u \tilde u_x +(R \tilde T\tilde n)_{x}=\tilde n \tilde{\phi}_x,\\
&\tilde u \tilde T_{x}+(\gamma-1)\tilde T\tilde u_x=0,\\
 &\tilde{\phi}_{xx}=\tilde n-e^{-\tilde{\phi}}.
\end{aligned}\right.
\end{eqnarray}
Here, corresponding to \eqref{1.5}-\eqref{1.6}, we also require that \eqref{1.7} is supplemented with
\begin{equation}\label{1.8}
\left\{\begin{aligned}
  & \inf_{x\in \mathbb{R}_{+}}\tilde n(x)>0,\ \ \
\inf_{x\in \mathbb{R}_{+}}\tilde T(x)>0,\ \ \   \left(\tilde{n}\tilde{u}+u_{e}e^{-\tilde{\phi}}\right)(0)=0,\\
&\lim_{x\rightarrow\infty}(\tilde n,\tilde u,\tilde T, \tilde
\phi)(x)=(1,u_{\infty},T_{\infty},0).
\end{aligned}\right.
\end{equation}
Note that
\begin{eqnarray}\label{1.gfb}
u_{e}e^{-\tilde{\phi}(0)}=-u_{\infty}=|u_{\infty}| \ \ or \ \ \ \tilde{\phi}(0)=ln\frac{u_{e}}{|u_{\infty}|}=:\phi_{b}
\end{eqnarray}
follows from \eqref{1.6}, $\eqref{1.7}_{1}$ and \eqref{1.8}. It is convenient to call $(\tilde n, \tilde
u, \tilde T, \tilde\phi)$ the boundary layer solution. Physically, this stationary boundary layer is also called a sheath.
In case $\phi_{b}=0$, if uniqueness is assumed, then one can only get the trivial solution $(\tilde n, \tilde
u, \tilde T, \tilde\phi)=(1,u_{\infty}, T_{\infty}, 0)$.
 Thus we consider the boundary layer solution under the assumption that $\phi_{b}\neq 0$, i.e., $|u_{\infty}|\neq u_{e}$.

In fact, to consider the existence of stationary solutions, the Sagdeev potential
\begin{equation}\label{1.9}
\left\{\begin{aligned}
& V(\phi):=\int_{0}^{\phi}
[f^{-1}(\eta)-e^{-\eta}]d\eta,\\
&\qquad\qquad\text{with }f(n)=\frac{\gamma
RT_{\infty}}{\gamma-1}\left(n^{\gamma-1}-1\right)+\frac{mu^{2}_{\infty}}{2 }\left(\frac{1}{n^{2}}-1\right)
\end{aligned}\right.
\end{equation}
plays a crucial role. One can compute that
\begin{eqnarray*}
&&\begin{aligned} &  f'(n)=\frac{-mu_{\infty}^{2}+\gamma
RT_{\infty}n^{\gamma+1}}{n^{3}}.
\end{aligned}
\end{eqnarray*}
Then, the only critical point of $f$ occurs at
\begin{equation*}
n=c_\infty:=\left(\frac{m
u_{\infty}^{2}}{\gamma RT_{\infty}}\right)^{\frac{1}{\gamma+1}},
\end{equation*}
where the constant $c_\infty$ is determined by the far-field data in connection with the Mach number at $x=\infty$. Therefore, in terms of the critical point $c_\infty$, the inverse function $f^{-1}$ in \eqref{1.9} is understood by adopting the branch which
contains the far-field equilibrium state $(\tilde{n}, \tilde{\phi})=(1, 0)$. Since the problem of the existence of
stationary solutions is reduced to the Dirichlet problem due to \eqref{1.gfb}, the unique existence
of the monotone stationary solution can be proved by a method similar to that in \cite{SIAm}.
Thus we omit the detailed discussions for brevity and list the main results in the following

\begin{proposition}[\cite{SIAm}]\label{prop1.1}
Consider the stationary problem \eqref{1.7} and \eqref{1.8} with the Dirichlet boundary condition
\begin{eqnarray}\label{mjo}
\tilde{\phi}(0)=\phi_{b}.
\end{eqnarray}
\begin{itemize}

\item[(i)] Let $u_{\infty}$ be a constant satisfying
$$
\text{either}\ \
u^{2}_{\infty}\leq \frac{\gamma RT_{\infty}}{m}
\ \ \text{or}\ \
\frac{\gamma
RT_{\infty}+1}{m}\leq u^{2}_{\infty}.
$$
Then the stationary problem
 has a unique monotone solution
$(\tilde n, \tilde u, \tilde T, \tilde\phi)$ verifying
\begin{eqnarray*}
&&\begin{aligned} &\tilde n,\tilde u, \tilde T, \tilde\phi \in
C(\overline{\mathbb{R}}_{+}),\ \ \ \tilde n, \tilde u, \tilde T,
\tilde \phi, \tilde\phi_{x} \in C^{1}(\mathbb{R}_{+})
\end{aligned}
\end{eqnarray*}
if and only if the boundary data $\phi_{b}$ satisfies conditions
\begin{eqnarray*}
&&V(\phi_{b})\geq 0,\ \ \  \phi_{b}\geq f(c_\infty).
\end{eqnarray*}

\item[(ii)]Let $u_{\infty}$ be a constant satisfying
$$
\frac{\gamma
RT_{\infty}}{m}<u^{2}_{\infty}<\frac{\gamma RT_{\infty}+1}{m}.
$$
If
$\phi_{b}\neq 0$, then the stationary problem does not admit any solutions in the function space
$ C^{1}(\mathbb{R}_{+}). $ If $\phi_{b}= 0$, then a constant
state $(\tilde n, \tilde u, \tilde T,
\tilde\phi)=(1,u_{\infty},T_{\infty},0)$ is the unique solution.
\end{itemize}
Moreover, the existing stationary solution enjoys some additional space-decay properties in the following two cases:
\begin{itemize}

\item (Nondegenerate case) Assume that
$$
\frac{\gamma RT_{\infty}+1}{m}<
u^{2}_{\infty},\quad {u_\infty<0,}
$$
and
$
\phi_{b}\neq f(c_\infty)
$
hold true. The
stationary solution $(\tilde n, \tilde u, \tilde T, \tilde\phi)$  belongs to
$C^{\infty}(\overline{\mathbb{R}}_{+})$ and verifies
\begin{eqnarray}\label{1.14}
&&\begin{aligned} &|\partial_{x}^{i}(\tilde
n-1)|+|\partial_{x}^{i}(\tilde
u-u_{\infty})|+|\partial_{x}^{i}(\tilde
T-T_{\infty})|+|\partial_{x}^{i}\tilde \phi|\leq C |
\phi_{b}|e^{-cx},
\end{aligned}
\end{eqnarray}
for any $i\geq 0$,
where $c$ and $C$ are positive constants.

\item (Degenerate case) Assume that
$$
\frac{\gamma RT_{\infty}+1}{m}= u^{2}_{\infty},\quad {u_\infty<0,}
$$
and
$
\phi_{b}>0
$
hold true. Denote constants
\begin{equation*}
\left\{\begin{aligned}
&c_{0}=1,\\
&c_{1}=-2\Gamma,\\
&c_{2}=\frac{(\gamma^{2}+\gamma)RT_{\infty}+2}{2},\\
&c_{3}=-2\Gamma[(\gamma^{2}+\gamma)RT_{\infty}+2],
\end{aligned}\right.
\end{equation*}
with
\begin{equation}
\label{def.conGa}
\Gamma=\sqrt{\frac{(\gamma^{2}+\gamma)RT_{\infty}+2}{12}}.
\end{equation}
There are constants $\delta_0>0$ and $C>0$ such that for any   $\phi_{b}\in (0,\delta_{0})$,
\begin{equation}
\label{1.14d}
\sum_{i=0}^3\|\partial _{x}^{i}UG^{i+2}+c_{i}\|_{L^\infty}\leq C
\phi_{b}
\end{equation}
with
\begin{equation*}
U=-\tilde{\phi}, \ \ \tilde{n}-1, \ \ \log \tilde{n}, \ \ \frac{\tilde{u}}{u_{\infty}}-1, \ \  \frac{1}{\gamma}\left(\frac{\tilde{T}}{T_{\infty}}-1\right),
\end{equation*}
where $G=G(x)$ is a function of the form
\begin{equation}
\label{def.Gx}
G(x)=\Gamma x+\phi_{b}^{-\frac{1}{2}}.
\end{equation}
\end{itemize}
\end{proposition}
As in \cite{FFC, Riemann}, we introduce the {\it Bohm criterion} that corresponds to the condition that
\begin{eqnarray}\label{1.15}
&&\begin{aligned} &u^{2}_{\infty}\geq \frac{\gamma
RT_{\infty}+1}{m},\quad u_{\infty}<0.
\end{aligned}
\end{eqnarray}
From Proposition \ref{prop1.1}, we see that under the Bohm criterion, there exists a unique monotone small-amplitude stationary solution provided that either
\begin{equation}
\label{1.15a}
|\phi_b|\ll 1,\quad u_\infty<-\sqrt{\frac{\gamma
RT_{\infty}+1}{m}}
\end{equation}
or
\begin{equation}
\label{1.15b}
0<\phi_b\ll 1,\quad u_\infty=-\sqrt{\frac{\gamma
RT_{\infty}+1}{m}}.
\end{equation}
In both cases we call the monotone stationary solution the plasma sheath. Thus we focus our attention on sheath with monotonicity.
From now on, we denote $(\tilde n, \tilde
u, \tilde T, \tilde\phi)$ to be the sheath solution to the half-space boundary-value problem \eqref{1.7}, \eqref{1.8} and \eqref{mjo} under the Bohm Criterion \eqref{1.15} additionally satisfying \eqref{1.15a} or \eqref{1.15b}.

The main concern of this paper is to study the asymptotic stability of the sheath $(\tilde n, \tilde
u, \tilde T, \tilde\phi)$. For this, it is convenient to
employ unknown functions $v:=\log n$ and $\tilde{v}:=\log \tilde{n}$
as well as perturbations
 $$(\varphi,\psi,\zeta,\sigma)(t,x)=(v,u,T,\phi)(t,x)-(\tilde v,\tilde u,\tilde T,\tilde \phi)(x).$$
From \eqref{1.4} and \eqref{1.7}, we have
\begin{eqnarray}\label{1.16}
\begin{aligned}[b] \left(\begin{array} {ccc}
1 & 0 & 0\\
0& m & 0\\
0& 0 & 1\\
\end{array} \right)
\left(\begin{array} {c}
\varphi \\
\psi\\
\zeta\\
\end{array} \right)_{t}
&+\left(\begin{array} {ccc}
u & 1 & 0\\
RT& mu & R\\
0& (\gamma-1)T & u\\
\end{array} \right)
\left(\begin{array} {c}
\varphi \\
\psi\\
\zeta\\
\end{array} \right)_{x}\\[2mm]
&=-\left(\begin{array} {ccc}
\psi & 0 & 0\\
R\zeta& m\psi & 0\\
0& (\gamma-1)\zeta & \psi\\
\end{array} \right)
\left(\begin{array} {c}
\tilde{v} \\
\tilde{u}\\
\tilde{T}\\
\end{array} \right)_{x}+\left(\begin{array} {c}
0 \\
\sigma_{x}\\
0\\
\end{array} \right),
\end{aligned}
\end{eqnarray}
and
\begin{eqnarray}\label{1.17}
\begin{aligned} &
\sigma_{xx}=e^{\varphi+\tilde{v}}-e^{\tilde{v}}-e^{-(\sigma+\tilde{\phi})}+e^{-\tilde{\phi}}.
\end{aligned}
\end{eqnarray}
The initial and  boundary data to \eqref{1.16}-\eqref{1.17} are
derived from  \eqref{1.5}-\eqref{1.6} and \eqref{1.8}-\eqref{1.gfb} as
\begin{eqnarray}\label{1.18}(\varphi,\psi,\zeta)(0,x)=(\varphi_{0},\psi_{0},\zeta_{0})(x):=(\log
n_{0}-\log \tilde n,u_{0}-\tilde u,T_{0}-\tilde T),
\end{eqnarray}
\begin{eqnarray}\label{1.19news}
  \lim_{x\rightarrow\infty}(\varphi_{0},\psi_{0},\zeta_{0})(x)=(0,0,0), \ \ \ \sigma_{x}(0,0)=r_{0}:=q_{0}-\tilde{\phi}_{x}(0),
\end{eqnarray}
\begin{eqnarray}\label{1.19}
  \sigma_{xt}(t,0)=\left[-e^{v}\psi-\tilde{u}(e^{v}-e^{\tilde{v}})+u_{\infty}(e^{-\sigma}-1)\right](t,0).
\end{eqnarray}

Provided that the perturbations are sufficiently small, both of the
characteristics of hyperbolic system  \eqref{1.16} are negative
owing to \eqref{1.15}, namely,
\begin{equation}\label{revad1}
\left\{\begin{aligned}
&\lambda_{1}[u,T]:=\frac{(m+1)u-\sqrt{(m-1)^{2}u^{2}+4\gamma RT}}{2}<0,\\
&\lambda_{2}[u,T]:=u<0,\\
&\lambda_{3}[u,T]:=\frac{(m+1)u+\sqrt{(m-1)^{2}u^{2}+4\gamma RT}}{2}<0.
\end{aligned}\right.
\end{equation}
Hence, no boundary conditions for the hyperbolic system \eqref{1.16}
are necessary for the well-posedness of the initial boundary value
problem \eqref{1.16}-\eqref{1.19}.

Before stating the main theorems, we first give the definition of the
function space $\mathscr{X}_{i}^{j}$ as follows:
$$
\mathscr{X}_{i}^{j}([0,M]):=\cap_{k=0}^{i}C^{k}([0,M];H^{j+i-k}(\mathbb{R}_{+})),
$$
$$
\mathscr{X}_{i}([0,M]):=\mathscr{X}_{i}^{0}([0,M]),
$$
for $i,j=0,1,2$, where $M>0$ is a constant.

We also show the stability of the sheath in both exponential and algebraic weighted Sobolev spaces.
The papers \cite{SIAm,boundaawa,Suzuki} pointed out that Euler-Poisson system \eqref{1.1} itself does not have the dissipative effect in the usual function space, however
there appear those effects in the weighted space. Therefore, we employ the weighted
space to do the energy estimates. The main theorems are stated as follows:

\begin{theorem}[Nondegenerate case]\label{1.2theorem}
Assume that the condition  $\eqref{1.15a}$ holds.

\medskip
(i) Assume that the initial data satisfy
$$
(e^{\frac{\lambda
x}{2}}\varphi_{0},e^{\frac{\lambda x}{2}}\psi_{0},e^{\frac{\lambda
x}{2}}\zeta_{0})\in (H^2(\mathbb{R}_{+}))^{3}
$$
for some positive constant $\lambda$.
Then there exists a positive constant $\delta$ such that if
$\beta\in(0,\lambda]$ and
$$
\beta+(|u_{e}/|u_{\infty}|-1|+\|(e^{\frac{\lambda
x}{2}}\varphi_{0},e^{\frac{\lambda x}{2}}\psi_{0},e^{\frac{\lambda
x}{2}}\zeta_{0})\|_{H^{2}})/\beta+|r_{0}|\leq\delta
$$
are satisfied, the initial
boundary value problem \eqref{1.16}-\eqref{1.19} has a unique
solution as
$$
(e^{\frac{\beta x}{2}}\varphi,e^{\frac{\beta
x}{2}}\psi,e^{\frac{\beta x}{2}}\zeta,e^{\frac{\beta
x}{2}}\sigma)\in (\mathscr{X}_{2}(\mathbb{R}_{+}))^{3} \times
\mathscr{X}_{2}^{2}(\mathbb{R}_{+}).
$$
Moreover, the solution
$(\varphi,\psi,\zeta,\sigma)$ verifies the decay estimate
\begin{eqnarray*}
\|(e^{\frac{\beta x}{2}}\varphi,e^{\frac{\beta
x}{2}}\psi,e^{\frac{\beta
x}{2}}\zeta)(t)\|_{H^{2}}^{2}+\|e^{\frac{\beta
x}{2}}\sigma(t)\|_{H^{4}}^{2} \leq C(\|(e^{\frac{\lambda
x}{2}}\varphi_{0},e^{\frac{\lambda x}{2}}\psi_{0},e^{\frac{\lambda
x}{2}}\zeta_{0})\|_{H^{2}}^{2}+r_{0}^{2})e^{-\mu t},
\end{eqnarray*}
where $C$ and $\mu$ are positive constants independent of $t$.

\medskip
(ii) Assume  $\lambda\geq 2$ holds. For an arbitrary
$\varepsilon\in(0,\lambda]$, there exists a positive constant
$\delta$ such that if
$$
((1+\beta
x)^{\frac{\lambda}{2}}\varphi_{0},(1+\beta
x)^{\frac{\lambda}{2}}\psi_{0},(1+\beta
x)^{\frac{\lambda}{2}}\zeta_{0})\in(H^2(\mathbb{R}_{+}))^{3}
$$
for $\beta>0$ and
$$
\beta+(|u_{e}/|u_{\infty}|-1|+\|((1+\beta
x)^{\frac{\lambda}{2}}\varphi_{0},(1+\beta
x)^{\frac{\lambda}{2}}\psi_{0},(1+\beta
x)^{\frac{\lambda}{2}}\zeta_{0})\|_{H^{2}})/\beta+|r_{0}|\leq\delta
$$
are satisfied, the initial boundary value problem
\eqref{1.16}-\eqref{1.19} has a unique solution as
$$
((1+\beta
x)^{\frac{\varepsilon}{2}}\varphi,(1+\beta
x)^{\frac{\varepsilon}{2}}\psi,(1+\beta
x)^{\frac{\varepsilon}{2}}\zeta,(1+\beta
x)^{\frac{\varepsilon}{2}}\sigma)\in
(\mathscr{X}_{2}(\mathbb{R}_{+}))^{3} \times
\mathscr{X}_{2}^{2}(\mathbb{R}_{+}).
$$
Moreover, the solution
$(\varphi,\psi,\zeta,\sigma)$ verifies the decay estimate
\begin{equation*}
\begin{aligned}
&\|((1+\beta x)^{\frac{\varepsilon}{2}}\varphi,(1+\beta
x)^{\frac{\varepsilon}{2}}\psi,(1+\beta
x)^{\frac{\varepsilon}{2}}\zeta)(t)\|_{H^{2}}^{2}+\|(1+\beta
x)^{\frac{\varepsilon}{2}}\sigma(t)\|_{H^{4}}^{2}\\
& \leq C(\|((1+\beta x)^{\frac{\lambda}{2}}\varphi_{0},(1+\beta
x)^{\frac{\lambda}{2}}\psi_{0},(1+\beta
x)^{\frac{\lambda}{2}}\zeta_{0})\|_{H^{2}}^{2}+r_{0}^{2})(1+\beta
t)^{-\lambda+\varepsilon},
\end{aligned}
\end{equation*}
where $C$ is a  positive
constant independent of $t$.
\end{theorem}

\begin{theorem}[Degenerate case]\label{1.3 theorem}
Assume that the condition  \eqref{1.15b} holds. Let
$4<\lambda_{0}<5.5693\cdots$ be the unique real solution to the
equation
\begin{equation}
\label{def.la0}
\lambda_{0}(\lambda_{0}-1)(\lambda_{0}-2)-12\left(\frac{2}{\gamma+1}\lambda_{0}+2\right)=0,
\end{equation}
where $5.5693\cdots $ is the unique real solution to the equation
\begin{equation}
\label{def.5.5}
\lambda_{0}(\lambda_{0}-1)(\lambda_{0}-2)-12(\lambda_{0}+2)=0.
\end{equation}
Assume that
$\lambda\in[4,\lambda_{0})$ is satisfied. For arbitrary
$\varepsilon\in(0,\lambda]$ and $\theta\in(0,1],$ there exists a
positive constant $\delta$ such that if
$\beta/(\Gamma \phi_{b}^{\frac{1}{2}} )\in [\theta,1]$,
$$
((1+\beta
x)^{\frac{\lambda}{2}}\varphi_{0},(1+\beta
x)^{\frac{\lambda}{2}}\psi_{0},(1+\beta
x)^{\frac{\lambda}{2}}\zeta_{0})\in(H^2(\mathbb{R}_{+}))^{3}
$$
and
$$
|u_{e}/|u_{\infty}|-1|+\|((1+\beta x)^{\frac{\lambda}{2}}\varphi_{0},(1+\beta
x)^{\frac{\lambda}{2}}\psi_{0},(1+\beta
x)^{\frac{\lambda}{2}}\zeta_{0})\|_{H^{2}}/\beta^{3}+|r_{0}|\leq\delta
$$
are satisfied, the initial boundary value problem
\eqref{1.16}-\eqref{1.19} has a unique solution as
$$
((1+\beta
x)^{\frac{\varepsilon}{2}}\varphi,(1+\beta
x)^{\frac{\varepsilon}{2}}\psi,(1+\beta
x)^{\frac{\varepsilon}{2}}\zeta,(1+\beta
x)^{\frac{\varepsilon}{2}}\sigma)\in
(\mathscr{X}_{2}(\mathbb{R}_{+}))^{3} \times
\mathscr{X}_{2}^{2}(\mathbb{R}_{+}).
$$
Moreover, the solution
$(\varphi,\psi,\zeta,\sigma)$ verifies the decay estimate
\begin{equation*}
\begin{aligned}
&\|((1+\beta x)^{\frac{\varepsilon}{2}}\varphi,(1+\beta
x)^{\frac{\varepsilon}{2}}\psi,(1+\beta
x)^{\frac{\varepsilon}{2}}\zeta)(t)\|_{H^{2}}^{2}+\|(1+\beta
x)^{\frac{\varepsilon}{2}}\sigma(t)\|_{H^{4}}^{2}\\
& \leq C(\|((1+\beta x)^{\frac{\lambda}{2}}\varphi_{0},(1+\beta
x)^{\frac{\lambda}{2}}\psi_{0},(1+\beta
x)^{\frac{\lambda}{2}}\zeta_{0})\|_{H^{2}}^{2}+r_{0}^{2})(1+\beta
t)^{-(\lambda-\varepsilon)/3},
\end{aligned}
\end{equation*}
where $C$ is a  positive
constant independent of $t$.
\end {theorem}

Since there are many occasions in which boundaries are insulated from other
conductors, in particular when we treat plasma flows outside laboratories. In such
situations, once charged particles are incident on the boundaries, they keep accumulating on the surface as expressed by \eqref{1.5b}. Then an interaction between the fluid
and the boundary occurs because the quantity of charged particles on the boundary
is proportional to the potential gradient there and hence has an influence on the electrostatic potential over the entire domain, which in turn affects the flux of charged
particles toward the boundary. To the best of our knowledge, no mathematical work has been done concerning the nonisentropic Euler-Poisson system with the fluid-boundary interaction.
For the isentropic Euler-Poisson system, Ohnawa \cite{boundaawa} studied the existence and asymptotic stability of boundary layers with the fluid-boundary interaction condition. Inspired by \cite{boundaawa}, we expect to consider the effect of the variable temperature for the nonisentropic Euler-Poisson system with the additional evolution equation of temperature function. In fact, in comparison with \cite{boundaawa}, we need to make additional efforts to consider the effect of the temperature equations in the proof.
Here we explain several crucial points in the proof of main results:

 $\bullet$
Technically, we observe that the only zero order dissipative term is associated with the weight parameter $\beta$, since this term arises from the integration by part of
the energy flux $I_{1}$ in \eqref{2.7}. From the physical viewpoint, the convection
can be said to stabilize the system. In terms of the property of the stationary solution
in Proposition \ref{prop1.1}, two integral terms $I_{1}$ and $I_{2}$ in \eqref{2.7} should be added together to estimate. In fact, to estimate $I_{1}+I_{2}$ in the degenerate case \eqref{1.15b}, a key point is to derive the positive definiteness of the quadratic form $Q(x)$ in \eqref{2.1gh} that takes a complex form. The same situation occurs to the proof of Lemma \ref{main.result2.987} for the nondegenerate case \eqref{1.15a}.

$\bullet$
 In the one-dimensional Dirichlet problem \cite{SIAm}, estimates up to the first order
derivatives are completed only by the same computations as in Lemma \ref{main.result2.9}. For
the present problem, however, Lemma \ref{main.result2.9} is not sufficient because a boundary term $\int_{0}^{t}(1+\beta
\tau)^{\xi}[\sigma^{2}+\sigma_{x}^{2}]_{x=0}d\tau$
appears on the right-hand side of \eqref{2.10}. To handle the boundary term, another
estimate is derived in Lemma \ref{main.resultloik0} by rewriting the fourth term of \eqref{2.1}. The idea was used
in the Dirichlet problem \cite{NOS}, but it was necessary only for two- or three-dimensional
cases. The idea was also used in the isentropic Euler-Poisson system with the fluid-boundary interaction problem \cite{boundaawa}.
In fact, the dissipative term on the boundary appears due to the fluid-boundary interaction.

$\bullet$ Technically, the estimate \eqref{2.10} has an integral term on the left-hand side while having a
boundary term on the other. The estimate \eqref{2.uykjui} has an opposite property. Thus we choose
two different weight functions as $(1+\beta x)^{\lambda}$ and $(1+\beta_{1} x)^{\lambda}$ $(\beta_{1}\ll\beta)$ to
lead to a favorable estimate.

The rest of the paper is arranged as follows. In Section 2, we give the energy estimates for the degenerate case. We make full use of the time-space weighted energy method to complete the proof of Theorem \ref{1.3 theorem}. In Section 3, we give the  energy estimates for the
nondegenerate case and complete the
proof of Theorem \ref{1.2theorem}. In the Appendix, we will give some basic results used in the proof of Proposition \ref{ste.pro1.1} and Proposition \ref{ste.pro2}.

\medskip
\noindent{\it Notations.} Throughout this paper, we denote a
positive constant (generally large) independent of $t$ by $C$. And
the character ``$C$" may take different values in different places.
 $L^p
= L^p(\mathbb{R}_{+}) \ (1 \leq p \leq \infty)$ denotes
 the usual Lebesgue space on $[0
,\infty)$ with its norm $ \|\cdot\|_ {L^p}$, and when $p=2$, we
write $ \| \cdot \| _{ L^2(\mathbb{R}_{+}) } = \| \cdot \|$. For a
nonnegative integer $s$,  $W^{s,p}$ denotes the usual $s$-th order
Sobolev space over $ [0,\infty)$ with its norm
$\|\cdot\|_{W^{s,p}}$. We use the abbreviation $H^s
(\mathbb{R}_{+})=W^{s,2}(\mathbb{R}_{+})$.
$C^{k}([0,T];H^s(\mathbb{R}_{+}))$ denotes the space of the $k$-times
continuously differential functions on the interval $[0,T]$ with
values in $H^s(\mathbb{R}_{+})$.
A norm with an algebraic weight is defined as follows:
$$\|f\|_{\alpha,\beta,i}:=\left(\int W_{\alpha,\beta}\sum_{j\leq
i}(\partial^{j}f)^{2}dx\right)^{\frac{1}{2}}, \ \ \ \ i,j\in
\mathbb{Z}, \ \ \ i,j\geq 0,$$
\begin{equation}
\label{def.W}
W_{\alpha,\beta}:=(1+\beta x)^{\alpha}, \ \ \ \ \alpha,\ \beta \in \mathbb{R},\ \ \ \beta>0.
\end{equation}
Note that this norm is equivalent to the norm defined by $\|(1+\beta
x)^{\frac{\alpha}{2}}f\|_{H^{i}}.$ The
 last subscript $i$ of $\|f\|_{\alpha,\beta,i}$ is often dropped for the case of $i=0$,
 namely, $\|f\|_{\alpha,\beta}:=\|f\|_{\alpha,\beta,0}$.

\section{Energy estimates in the degenerate case}
In this section, we study the asymptotic stability of the
sheath to \eqref{1.1} for the degenerate case
\eqref{1.15b}, where the Bohm criterion is marginally fulfilled. In this case, we see from Proposition \ref{prop1.1} that the additional condition that $\phi_{b}>0$ is suitably small ensures the existence of a non-trivial monotone stationary solution to \eqref{1.7}-\eqref{1.gfb}. To further show the dynamical stability of the sheath, we mainly focus on the {\it a priori} estimates that will be given in Proposition \ref{ste.pro1.1}. The global existence can be proved by the standard continuation argument based on the local
existence result together with the uniform {\it a priori} estimates. The local-in-time existence result can be proved by a similar method as in \cite{boundaawa,SIAm} and we omit the  details of the proof for brevity.

In what follows we are devoted to establishing the {\it a priori} estimates in the degenerate case \eqref{1.15b}. For this purpose, we use the following notation for
convenience
\begin{eqnarray*}
 \mathcal
{N}_{\alpha,\beta}(M):=\sup_{0\leq t\leq
M}(\|(\varphi,\psi,\zeta)(t)\|_{\alpha,\beta,2}+|\sigma_{x}(t,0)|).
\end{eqnarray*}

\begin{proposition}\label{ste.pro1.1}
Let the same conditions on
$T_{\infty}$, $u_{\infty}$, $\lambda_{0}$ and $\lambda$ as in
Theorem \ref{1.3 theorem} hold and let $(\varphi,\psi,\zeta,\sigma)$ be a solution to
\eqref{1.16}-\eqref{1.19} over $[0,M]$ for $M>0$.
For any $\varepsilon\in(0,\lambda]$ and any $\theta\in(0,1]$,
there exist constants $\delta>0$ and $C>0$ independent of $M$ such that if all the following conditions
\begin{eqnarray}\label{b}
\beta/(\Gamma \phi_{b}^{1/2}) \in [\theta,1],
\end{eqnarray}
\begin{eqnarray}\label{c}
((1+\beta x)^{\frac{\lambda}{2}}\varphi,(1+\beta
x)^{\frac{\lambda}{2}}\psi,(1+\beta
x)^{\frac{\lambda}{2}}\zeta,(1+\beta x)^{\frac{\lambda}{2}}\sigma)
\in (\mathscr{X}_{2}([0,M]))^{3}\times \mathscr{X}_{2}^{2}([0,M]),
\end{eqnarray}
and
\begin{eqnarray}\label{d}
\phi_{b}+\mathcal {N}_{\lambda,\beta}(M)/\beta^{3}\leq \delta
\end{eqnarray}
are satisfied, then it holds for any $0\leq t\leq M$
that
\begin{eqnarray}
\|(\varphi,\psi,\zeta)(t)\|_{\varepsilon,\beta,2}^{2}+\|\sigma(t)\|_{\varepsilon,\beta,4}^{2}\leq
C(\|(\varphi_{0},\psi_{0},\zeta_{0})\|_{\lambda,\beta,2}^{2}+r_{0}^{2})(1+\beta
t)^{-(\lambda-\varepsilon)/3}.\label{prop2.1r1}
\end{eqnarray}
\end{proposition}

 For the proof of Proposition \ref{ste.pro1.1}, only the algebraic weight technique seems to be effective to the degenerate problem. We need to first prove Lemmata \ref{main.result2.9}  and  \ref{main.resultloik0} which are crucial steps for deriving the {\it a priori} estimates on the zeroth order and first order space derivatives. After that, we give the estimates for the higher order derivatives in
 Lemmata \ref{main.result2.erer9} and \ref{mainboundary2g}. Proposition \ref{ste.pro1.1} is then proved by following Lemmata $\ref{main.result2.9}$--$\ref{mainboundary2g}$ at the end of this section.

\begin{lemma}\label{main.result2.9}
Under the same conditions as in Proposition \ref{ste.pro1.1}, there exist
positive constants $C$ and $\delta$ independent of $M$ such that if  conditions \eqref{b}-\eqref{d} are satisfied, it holds for
any $t\in[0,M]$ and any $\xi\geq 0$ that
\begin{align}
 &(1+\beta t)^{\xi}\|(\varphi,\psi,\zeta)(t)\|_{\varepsilon,\beta,1}^{2}\notag\\
&\quad +\int_{0}^{t} (1+\beta
 \tau)^{\xi}\left\{\beta^{3}\|(\varphi,\psi,\zeta)\|_{\varepsilon-3,\beta}^{2}
 +\beta\|(\varphi_{x},\psi_{x},\zeta_{x},\sigma_{x})\|_{\varepsilon-1,\beta}^{2}\right\}d\tau\notag\\
& \leq C\|(\varphi_{0},\psi_{0},\zeta_{0})\|_{\varepsilon,\beta,1}^{2}
 +C\xi\beta\int_{0}^{t} (1+\beta
 \tau)^{\xi-1}\|(\varphi,\psi,\zeta)\|_{\varepsilon,\beta,1}^{2}
 d\tau\notag\\
 &\quad+C\int_{0}^{t}(1+\beta
\tau)^{\xi}[\sigma^{2}+\sigma_{x}^{2}]_{x=0}d\tau.
\label{2.10}
 \end{align}

\end{lemma}

\begin{proof}
We start to derive from \eqref{1.16} and \eqref{1.17} several identities which will are used in the late energy estimates. First, it is convenient to rewrite \eqref{1.16} as
\begin{eqnarray}\label{2.0}
\begin{aligned}[b]
\left(\begin{array} {ccc}
RT \ \ & 0 & \ \ 0\\
0 \ \ & m & \ \ 0\\
0 \ \ & 0 & \ \ \frac{R}{(\gamma-1)T} \\
\end{array} \right)
\left(\begin{array} {c}
\varphi \\
\psi\\
\zeta\\
\end{array} \right)_{t}&+\left(\begin{array} {ccc}
RTu \ \ & RT & \ \ 0\\
RT \ \ & mu & \ \ R\\
0 \ \ & R & \ \ \frac{R u}{(\gamma-1)T}\\
\end{array} \right)
\left(\begin{array} {c}
\varphi \\
\psi\\
\zeta\\
\end{array} \right)_{x}-\left(\begin{array} {c}
0 \\
\sigma\\
0\\
\end{array} \right)_{x}\\[2mm]
&=-\left(\begin{array} {ccc}
RT\psi \ \ & 0 & \ \ 0\\
R\zeta \ \ & m\psi & \ \ 0\\
0 \ \ & \frac{R\zeta}{T} & \ \ \frac{R\psi}{(\gamma-1)T}\\
\end{array} \right)
\left(\begin{array} {c}
\tilde{v} \\
\tilde{u}\\
\tilde{T}\\
\end{array} \right)_{x}.
\end{aligned}
\end{eqnarray}
Taking the inner product of \eqref{2.0} with $\tilde{n}(\varphi,\psi,\zeta)$ and using $\tilde{v}_{x}=\frac{\tilde{n}_{x}}{\tilde{n}}$, one can get that
\begin{equation}
\label{2.1}
(\CE_0)_t+(\CH_0)_x+\CD_0+\tilde{n}\psi_{x}\sigma=\CR_0,
\end{equation}
where we have denoted
\begin{equation}
\label{def.e0}
\CE_0=\frac{\tilde{n}}{2}RT\varphi^{2}+\frac{\tilde{n}}{2}m\psi^{2}+\frac{\tilde{n}R}{2(\gamma-1)T}\zeta^{2},
\end{equation}
\begin{equation*}
\CH_0=\frac{\tilde{n}}{2}RTu\varphi^{2}+\tilde{n}RT\varphi\psi+\frac{\tilde{n}}{2}mu\psi^{2}+R\tilde{n}\zeta\psi+\frac{\tilde{n}R
u}{2(\gamma-1)T}\zeta^{2}
-\tilde{n}\sigma\psi,
\end{equation*}
\begin{eqnarray*}
\CD_0 & = & \left(-\frac{R T
u}{2}\tilde{n}_{x}-\frac{R \tilde{n}u}{2}\tilde{T}_{x}-\frac{R
\tilde{n}T}{2}\tilde{u}_{x}\right)\varphi^{2}-\tilde{n}R\tilde{T}_{x}\varphi\psi
+\left(\frac{m\tilde{n}}{2}\tilde{u}_{x}-\frac{mu}{2}\tilde{n}_{x}\right)\psi^{2}\\[2mm]
&&+\frac{R\tilde{n}}{(\gamma-1)T}\tilde{T}_{x}\zeta\psi+\tilde{n}_{x}\sigma\psi+\left(\frac{R\tilde{n}}{T}\tilde{u}_{x}
-\frac{R u
\tilde{n}_{x}+R\tilde{n}\tilde{u}_{x}}{2(\gamma-1)T}+\frac{R u
\tilde{n} \tilde{T}_{x}}{2(\gamma-1)T^{2}}\right)\zeta^{2},
\end{eqnarray*}
and
\begin{eqnarray*}
\CR_0 &=&\left(\frac{\tilde{n}R}{2}\zeta_{t}+\frac{\tilde{n}R
u}{2}\zeta_{x}+\frac{\tilde{n}R
T}{2}\psi_{x}\right)\varphi^{2}+R\tilde{n}\zeta_{x}\varphi\psi+\frac{m\tilde{n}}{2}\psi_{x}\psi^{2}\notag\\[2mm]
&&+\left(\frac{R\tilde{n}\psi_{x}}{2(\gamma-1)T}-\frac{R\tilde{n}\zeta_{t}}{2(\gamma-1)T^{2}}-
\frac{R\tilde{n}u\zeta_{x}}{2(\gamma-1)T^{2}}\right)\zeta^{2}.
\label{def.r0}
\end{eqnarray*}
Taking the one order $x$-derivative on \eqref{2.0}, further taking the inner product of the resulting system with $\tilde{n}(\varphi_{x},\psi_{x},\zeta_{x})$ and using $\tilde{v}_{x}=\frac{\tilde{n}_{x}}{\tilde{n}}$ again, similarly for obtaining \eqref{2.1}, one has
\begin{equation}
\label{2.2}
(\CE_1^{\rm x})_t+(\CH_1^{\rm x})_x-\tilde{n}\psi_{x}\sigma_{xx}=\CR_1^{\rm x},
\end{equation}
where we also have denoted
\begin{equation}
\label{def.e1x}
\CE_1^{\rm x}=\frac{\tilde{n}}{2}RT\varphi_{x}^{2}+\frac{\tilde{n}}{2}m\psi_{x}^{2}+\frac{\tilde{n}R}{2(\gamma-1)T}\zeta_{x}^{2},
\end{equation}
\begin{equation}
\label{def.h1x}
\CH_1^{\rm x}=\frac{\tilde{n}}{2}RTu\varphi_{x}^{2}+\tilde{n}RT\varphi_{x}\psi_{x}+\frac{\tilde{n}}{2}mu\psi_{x}^{2}+R\tilde{n}\zeta_{x}\psi_{x}+\frac{\tilde{n}R
u}{2(\gamma-1)T}\zeta_{x}^{2},
\end{equation}
and
\begin{eqnarray}
\CR_1^{\rm x}&=&\left[\frac{\tilde{n}R}{2}\zeta_{t}+ \frac{R T
u}{2}\tilde{n}_{x}-\frac{R \tilde{n}}{2}(T u)_{x}\right]\varphi_{x}^{2}
+\left[\frac{m
u}{2}\tilde{n}_{x}-\frac{m\tilde{n}}{2}u_{x}-m\tilde{n} \tilde{u}
_{x} \right]\psi_{x}^{2}\notag\\[2mm]
&&+\left[ \frac{R
u}{2(\gamma-1)T}\tilde{n}_{x}+\frac{R
\tilde{n}u}{2(\gamma-1)T^{2}}T_{x}-\frac{\tilde{n}R}{2(\gamma-1)T^{2}}\zeta_{t}-\frac{R
\tilde{n}}{2(\gamma-1)T}u_{x}-\frac{R
\tilde{n}}{T}\tilde{u}_{x}\right]\zeta_{x}^{2}\notag\\[2mm]
&&
 -R\tilde{n}T_{x}(\varphi_{t}+\psi
_{x})\varphi_{x}-R\tilde{n}_{x}\psi
 T_{x}\varphi_{x}+\frac{R\tilde{n}}{(\gamma-1)T^{2}}\zeta_{t}\zeta_{x}T_{x}
 +\frac{R\tilde{n}}{T^{2}}\tilde{u}_{x}\zeta\zeta_{x}T_{x}\notag\\[2mm]
&&+\frac{R\tilde{n}}{(\gamma-1)T^{2}}\tilde{T}_{x}\psi\zeta_{x}T_{x}-
\frac{R\tilde{n}}{(\gamma-1)T}\tilde{T}_{x}\psi_{x}\zeta_{x}
-R\tilde{n}\tilde{v}_{xx}(T\psi\varphi_{x}+\zeta\psi_{x})\notag\\[2mm]
&&-\tilde{n}\tilde{u}_{xx}(m\psi\psi_{x}+\frac{R}{T}\zeta\zeta_{x})
-\frac{R\tilde{n}}{(\gamma-1)T}\tilde{T}_{xx}\psi\zeta_{x}.
\label{def.r1x}
\end{eqnarray}

In what follows, we are going to multiply \eqref{2.1} by $e^{-\tilde{\phi}}$ and add the resulting equation together with \eqref{2.2}. To treating terms involving $\si$ in this process, we first notice that
 \begin{equation}
\label{2.3p1}
\tilde{n}\psi_{x}\sigma\cdot e^{-\tilde{\phi}}+(-\tilde{n}\psi_{x}\sigma_{xx})=\tilde{n}(e^{-\tilde{\phi}}\sigma-\sigma_{xx})\psi_{x}.
\end{equation}
Recall \eqref{1.17}. It follows from the Taylor expansion that
\begin{eqnarray}\label{2.3}
 \begin{aligned}
\sigma_{xx}=\tilde{n}\left(\varphi+\frac{1}{2}e^{\theta_{1}\varphi}\varphi^{2}\right)+e^{-\tilde{\phi}}\left(\sigma-\frac{1}{2}e^{-\theta_{2}\sigma}\sigma^{2}\right),\
\ \  \theta_{1},\theta_{2}\in(0,1).
\end{aligned}
\end{eqnarray}
Substituting the above form into the right-hand side of \eqref{2.3p1} and using the first component equation of \eqref{1.16}, one has
\begin{eqnarray}
\tilde{n}\psi_{x}\sigma\cdot e^{-\tilde{\phi}}+(-\tilde{n}\psi_{x}\sigma_{xx})&=&\left(\frac{1}{2}\tilde{n}^{2}\varphi^{2}\right)_{t}+
\left(\frac{1}{2}\tilde{n}^{2}u\varphi^{2}\right)_{x}-\frac{1}{2}\tilde{n}^{2}\tilde{u}_{x}\varphi^{2}-\tilde{n}u\tilde{n}_{x}\varphi^{2}
+\tilde{n}^{2}\tilde{v}_{x}\varphi\psi\notag\\[2mm]
&&-\frac{1}{2}\tilde{n}^{2}\psi_{x}\varphi^{2}-\frac{\tilde{n}}{2}\left(\tilde{n}e^{\theta_{1}\varphi}\varphi^{2}
-e^{-(\theta_{2}\sigma+\tilde{\phi})}\sigma^{2}\right)\psi_{x}.
\label{2.4}
\end{eqnarray}
Therefore, taking the procedure \eqref{2.1}$\times e^{-\tilde{\phi}}+$\eqref{2.4} gives that
\begin{eqnarray}
 &\dis\left(e^{-\tilde{\phi}}\CE_0
+\CE_1^{\rm x}+\frac{1}{2}\tilde{n}^{2}\varphi^{2}\right)_{t}
+\left(e^{-\tilde{\phi}}\CH_0+\CH_1^{\rm x}
+\frac{1}{2}\tilde{n}^{2}u\varphi^{2}\right)_{x}
+e^{-\tilde{\phi}}\tilde{\phi}_{x}\CH_0+e^{-\tilde{\phi}}\CD_0
\notag\\
&\dis+\left(-\tilde{n}u\tilde{n}_{x}-\frac{1}{2}\tilde{n}^{2}\tilde{u}_{x}
\right)\varphi^{2}+\tilde{n}^{2}\tilde{v}_{x}\varphi\psi =\CN_1,
\label{2.5}
\end{eqnarray}
where
\begin{equation}\label{2.6}
\CN_{1}=e^{-\tilde{\phi}}\CR_0+\CR_1^{\rm x}
+\frac{1}{2}\tilde{n}^{2}\psi_{x}\varphi^{2}+\frac{\tilde{n}}{2}\left(\tilde{n}e^{\theta_{1}\varphi}\varphi^{2}
-e^{-(\theta_{2}\sigma+\tilde{\phi})}\sigma^{2}\right)\psi_{x}.
\end{equation}

Recall \eqref{def.la0} and \eqref{def.5.5} for the definition of $\la_0$. Let $\la\in [4,\la_0)$ and then $\varepsilon\in (0,\la]$ be given.  We choose a space weight function $W_{\varepsilon,\beta}=(1+\beta x)^\varepsilon$ as in \eqref{def.W} for a suitable parameter $\beta>0$ depending on $\varepsilon$ to be determined later. Then, multiplying \eqref{2.5} by $ W_{\varepsilon,\beta}$ and
integrating the resulting equation over $\mathbb{R}_{+}$,
one deduces that
\begin{eqnarray}\label{2.7}
\begin{aligned}[b]
&\frac{d}{dt}\int_{\mathbb{R}_{+}} W_{\varepsilon,\beta}
\left[e^{-\tilde{\phi}}\CE_0
+\CE_1^{\rm x}+\frac{1}{2}\tilde{n}^{2}\varphi^{2}\right]dx\\[2mm]
&\quad+\underbrace{\int_{\mathbb{R}_{+}}\varepsilon\beta W_{\varepsilon-1,\beta}\CL_1dx}_{I_{1}}
 +\int_{\mathbb{R}_{+}}\varepsilon\beta
W_{\varepsilon-1,\beta}(-\CH_1^{\rm x})
dx
+\underbrace{\int_{\mathbb{R}_{+}} W_{\varepsilon,\beta}\CL_{2} dx}_{I_{2}}\\[2mm]
&\quad-[e^{-\tilde{\phi}}\CH_0
+\frac{1}{2}\tilde{n}^{2}u\varphi^{2}](t,0)-\CH_1^{\rm x}(t,0)=\int_{\mathbb{R}_{+}} W_{\varepsilon,\beta}\CN_{1} dx,
\end{aligned}
\end{eqnarray}
where
\begin{equation*}
\CL_1=
-e^{-\tilde{\phi}}\CH_0-\frac{1}{2}\tilde{n}^{2}u\varphi^{2}
\end{equation*}
and
\begin{eqnarray*}
\CL_{2}&=&e^{-\tilde{\phi}}\tilde{\phi}_{x}\CH_0+e^{-\tilde{\phi}}\CD_0+\left(-\tilde{n}u\tilde{n}_{x}-\frac{1}{2}\tilde{n}^{2}\tilde{u}_{x}
\right)\varphi^{2}+\tilde{n}^{2}\tilde{v}_{x}\varphi\psi.
\end{eqnarray*}

Now we estimate each term in \eqref{2.7} and we shall frequently use the results obtained in Section \ref{Append}. First, one can decompose $u$ and $T$ as $u=\psi+(\tilde{u}-u_{\infty})+u_{\infty}$ and  $T=\zeta+(\tilde{T}-T_{\infty})+T_{\infty}$, respectively. Recall \eqref{1.14d}
and Lemma \ref{main.result4.2}, as well as \eqref{def.h1x} for $\CH_1^{\rm x}$. Then one sees, under the condition \eqref{1.15b} and
 \eqref{b}-\eqref{d}, that
\begin{eqnarray}\label{2.10a}
 \begin{aligned}[b]
&\frac{-\CH_1^{\rm x}}{\tilde{n}}=-\frac{1}{2}RTu\varphi_{x}^{2}-RT\varphi_{x}\psi_{x}-\frac{1}{2}mu\psi_{x}^{2}-R\zeta_{x}\psi_{x}-\frac{R
u}{2(\gamma-1)T}\zeta_{x}^{2}\\[2mm]
&\geq\frac{1}{2}RT_{\infty}(-u_{\infty})\varphi_{x}^{2}-RT_{\infty}\varphi_{x}\psi_{x}+\frac{1}{2}m(-u_{\infty})\psi_{x}^{2}
-R\zeta_{x}\psi_{x}+\frac{R
(-u_{\infty})}{2(\gamma-1)T_{\infty}}\zeta_{x}^{2}\\[2mm]
&\qquad-C(\mathcal
{N}_{\lambda,\beta}(M)+\phi_{b})(\varphi_{x}^{2}+\psi_{x}^{2}+\zeta_{x}^{2})\\[2mm]
&\geq (c-C\delta)(\varphi_{x}^{2}+\psi_{x}^{2}+\zeta_{x}^{2}).
\end{aligned}
\end{eqnarray}
Therefore, it follows that the third term on the left-hand side of \eqref{2.7} can be estimated as
\begin{equation}\label{2.11a}
\int_{\mathbb{R}_{+}}\varepsilon\beta
W_{\varepsilon-1,\beta} (-\CH_1^{\rm x})dx
\geq
c\beta\|(\varphi_{x},\psi_{x},\zeta_{x})\|_{\varepsilon-1,\beta}^{2}.
\end{equation}
Here and in the sequel we have omitted the explicit dependence of $c>0$ on $\varepsilon$ for brevity and instead we would only emphasize the dependence of the constant coefficient on $\beta$.

In the same way as for treating \eqref{2.10a}, with the help of $u_{\infty}<0$ as well as the smallness of $\delta>0$, for the boundary terms on the left-hand side of \eqref{2.7} and under the condition \eqref{1.15b}, one has
\begin{equation}\label{2.11}
-\CH_1^{\rm x}(t,0)\geq (c-C\delta)[\varphi_{x}^{2}+\psi_{x}^{2}+\zeta_{x}^{2}]_{x=0}
\end{equation}
and
\begin{eqnarray}\label{2.11b}
 \begin{aligned}[b]
-\left[e^{-\tilde{\phi}}\CH_0
+\frac{1}{2}\tilde{n}^{2}u\varphi^{2}\right](t,0)
\geq (c-C\delta)[\varphi^{2}+\psi^{2}+\zeta^{2}]_{x=0}-(\frac{|u_{\infty}|}{2}+C\delta)[\sigma^{2}]_{x=0}.
\end{aligned}
\end{eqnarray}

It remains to estimate two terms $I_{1}$ and
$I_{2}$ on the left-hand side of \eqref{2.7}. The
key is to make full use of properties of the stationary solution in \eqref{1.14d}.  Through careful computations, one can capture the full energy dissipation of all the zero-order components with the positive coefficient. In fact, using \eqref{1.14d} and Lemma \ref{main.result4.2} together with the identity that $\tilde{n}(x)\tilde{u}(x)\equiv u_{\infty}$ and recalling $G=G(x)=\Ga x +\phi_b^{-1/2}$ as in \eqref{def.Gx} with the constant $\Gamma$ defined in \eqref{def.conGa}, one has
\begin{align}
 I_{1} &\geq\int_{\mathbb{R}_{+}}\varepsilon\beta
W_{\varepsilon-1,\beta}\Big\{
\frac{1-G^{-2}}{2}(RT_{\infty}+1)|u_{\infty}|\varphi^{2}-(1-2G^{-2})RT_{\infty}\varphi\psi\notag\\[2mm]
&\qquad\qquad\qquad\qquad+\frac{1-G^{-2}}{2}m|u_{\infty}|\psi^{2}
+(1-G^{-2})\frac{R
|u_{\infty}|}{2(\gamma-1)T_{\infty}}\zeta^{2}\notag\\[2mm]
&\qquad\qquad\qquad\qquad-(1-2G^{-2})R\zeta\psi+(1-2G^{-2})\sigma\psi\Big\}dx\notag\\[2mm]
&\quad
-C\mathcal
{N}_{\lambda,\beta}(M)\|(\varphi,\psi,\zeta,\sigma)\|^{2}_{\varepsilon-3,\beta}\notag\\[2mm]
&\quad-C\phi_{b}\int_{\mathbb{R}_{+}}\beta
W_{\varepsilon-1,\beta}G^{-2}(\varphi^{2}+\psi^{2}+\zeta^{2}+\sigma^{2})dx,\label{2.12}
\end{align}
and
\begin{align}
 I_{2}\geq&\int_{\mathbb{R}_{+}}
W_{\varepsilon,\beta}G^{-3}\Gamma|u_{\infty}|\left\{(\gamma
RT_{\infty}+1)\varphi^{2}+\frac{2(1-\gamma
RT_{\infty})}{|u_{\infty}|}\varphi\psi+3m\psi^{2}\right.\notag\\[2mm]
&\qquad\qquad\qquad\qquad\qquad\qquad\left.+\frac{4}{|u_{\infty}|}\sigma\psi+\frac{\gamma
R}{(\gamma-1)T_{\infty}}\zeta^{2}\right\}dx\notag\\[2mm] &
-C(\mathcal {N}_{\lambda,\beta}(M)+\phi_{b})\int_{\mathbb{R}_{+}}
W_{\varepsilon,\beta}G^{-3}(\varphi^{2}+\psi^{2}+\zeta^{2}+\sigma^{2})dx.\label{2.13}
\end{align}
Adding \eqref{2.12} to  \eqref{2.13} together and further using the Cauchy-Schwarz inequality
$$
\sigma\psi\geq-\left(\frac{|u_{\infty}|}{2}\sigma^{2}+\frac{1}{2|u_{\infty}|}\psi^{2}\right)
$$
and the condition \eqref{1.15b}, one has
\begin{align}
I_{1}+I_{2} &\geq I_{1,2}-C\mathcal
{N}_{\lambda,\beta}(M)\|(\varphi,\psi,\zeta,\sigma)\|^{2}_{\varepsilon-3,\beta}-C\phi_{b}\int_{\mathbb{R}_{+}}\beta
W_{\varepsilon-1,\beta}G^{-2}(\varphi^{2}+\psi^{2}+\zeta^{2}+\sigma^{2})dx\notag\\
&\quad-C(\mathcal
{N}_{\lambda,\beta}(M)+\phi_{b})\int_{\mathbb{R}_{+}}
W_{\varepsilon,\beta}G^{-3}(\varphi^{2}+\psi^{2}+\zeta^{2}+\sigma^{2})dx,\label{2.14}
\end{align}
where we have defined
\begin{eqnarray}\label{2.15}
 \begin{aligned}[b]
I_{1,2}=&\int_{\mathbb{R}_{+}}\left\{\frac{\varepsilon\beta}{2}W_{\varepsilon-1,\beta}(1-G^{-2})(RT_{\infty}+1)|u_{\infty}|+\Gamma|u_{\infty}|(\gamma
RT_{\infty}+1)W_{\varepsilon,\beta}G^{-3}\right\}\varphi^{2}dx\\[2mm] &
+\int_{\mathbb{R}_{+}}\left\{-RT_{\infty} \varepsilon\beta
W_{\varepsilon-1,\beta}(1-2G^{-2})+2\Gamma(1-\gamma
RT_{\infty})W_{\varepsilon,\beta}G^{-3} \right\}\varphi \psi dx\\[2mm] &
+
\frac{1}{|u_{\infty}|}\int_{\mathbb{R}_{+}}\left\{W_{\varepsilon-1,\beta}\frac{\varepsilon\beta}{2}[\gamma
RT_{\infty}+(1-\gamma RT_{\infty})G^{-2}]+\Gamma(3\gamma RT_{\infty}+1)W_{\varepsilon,\beta}G^{-3}\right\}\psi^{2}dx\\[2mm] &
-\frac{|u_{\infty}|}{2}\int_{\mathbb{R}_{+}}\left\{\varepsilon\beta
W_{\varepsilon-1,\beta}(1-2G^{-2})+4\Gamma W_{\varepsilon,\beta}G^{-3} \right\}\sigma^{2} dx\\[2mm] &
+\int_{\mathbb{R}_{+}}\left\{\frac{\varepsilon\beta}{2}W_{\varepsilon-1,\beta}\frac{R|u_{\infty}|}{(\gamma-1)T_{\infty}}(1-G(x)^{-2})
+\frac{\Gamma\gamma R|u_{\infty}|}{(\gamma-1)T_{\infty}}W_{\varepsilon,\beta}G^{-3}\right\}\zeta^{2}dx\\[2mm] &
-\int_{\mathbb{R}_{+}}\varepsilon\beta W_{\varepsilon-1,\beta} R(1-2G^{-2})\zeta
\psi dx.
\end{aligned}
\end{eqnarray}

Now we {\it claim} a key estimate on the coercivity of $I_{1,2}$ as follows:
\begin{eqnarray}\label{2.16}
 \begin{aligned}[b]
I_{1,2}\geq c \beta^{3}
\|(\varphi,\psi,\zeta)\|_{\varepsilon-3,\beta}^{2}+c\beta\|\sigma_{x}\|_{\varepsilon-1,\beta}^{2}-C\beta[\sigma^{2}+\sigma_{x}^{2}]_{x=0},
\end{aligned}
\end{eqnarray}
where as mentioned before, the constant $c>0$ may depend on $\varepsilon$ but not on $\beta$. Indeed, multiplying \eqref{2.3} by $-\varepsilon\beta \sigma
W_{\varepsilon-1,\beta} $ and integrating the resulting equation over $\mathbb{R}_{+}$ with the help of \eqref{1.14d} and the Cauchy-Schwarz
inequality, it follows that
\begin{align*}
&\int_{\mathbb{R}_{+}} \varepsilon\beta
W_{\varepsilon-1,\beta}\left\{\sigma_{x}^{2}+\frac{1}{2}(1-2G^{-2})\sigma^{2}
\right\}dx\notag \\[2mm]
&\leq
\int_{\mathbb{R}_{+}}\frac{\varepsilon\beta}{2}W_{\varepsilon-1,\beta}(1-2G^{-2})\varphi^{2}dx+\int_{\mathbb{R}_{+}}\frac{1}{2}\varepsilon(\varepsilon-1)(\varepsilon-2)
\beta^{3}W_{\varepsilon-3,\beta}\varphi^{2}dx\notag\\[2mm]
&\quad+C(\beta^{2}+\phi_{b}+\mathcal{N}_{\lambda,\beta}(M)\beta^{-2})\beta^{3}\|\varphi\|_{\varepsilon-3,\beta}^{2}+C\beta[\sigma^{2}+\sigma_{x}^{2}]_{x=0}.
\end{align*}
Applying the above estimate into the fourth term on the right-hand side of \eqref{2.15}, we are able to obtain
\begin{align}
I_{1,2}&\geq  \int_{\mathbb{R}_{+}}\beta
W_{\varepsilon-1,\beta}Q(x)dx+|u_{\infty}|\varepsilon\beta\|\sigma_{x}\|_{\varepsilon-1,\beta}^{2}\notag\\[2mm]
&\quad-C(\beta^{2}+\phi_{b}+\mathcal{N}_{\lambda,\beta}(M)\beta^{-2})\beta^{3}\|\varphi\|_{\varepsilon-3,\beta}^{2}-C\beta[\sigma^{2}+\sigma_{x}^{2}]_{x=0}, \label{2.18}
\end{align}
where $Q(x)$  is a quadratic form of $\varphi,$ $\psi$ and  $\zeta$ defined by
\begin{eqnarray}\label{2.1gh}
Q(x)=|u_{\infty}|q_{1}(x)\varphi^{2}+q_{2}(x)\varphi\psi+\frac{1}{|u_{\infty}|}q_{3}(x)\psi^{2}
+|u_{\infty}|q_{4}(x)\zeta^{2}+q_{5}(x)\zeta\psi,
\end{eqnarray}
with
\begin{eqnarray*}
q_{1}(x)&=&\frac{\varepsilon}{2}RT_{\infty}+B(x)^{-2}\Gamma^{-2}\Big\{\frac{(1-RT_{\infty})\varepsilon}{2}S(x)^{2}
+(\gamma
RT_{\infty}-1)S(x)^{3}\\[2mm]
&&\qquad\qquad\qquad\qquad\qquad-\frac{\Gamma^{2}}{2}\varepsilon(\varepsilon-1)(\varepsilon-2)
\Big\},\\[2mm]
q_{2}(x)&=&-RT_{\infty}\varepsilon+B(x)^{-2}\Gamma^{-2}\left\{2\varepsilon
RT_{\infty} S(x)^{2} +2(1-\gamma RT_{\infty})S(x)^{3} \right\},\\[2mm]
q_{3}(x)&=&\frac{\varepsilon}{2}\gamma
RT_{\infty}+B(x)^{-2}\Gamma^{-2}\left\{\frac{(1-\gamma
RT_{\infty})\varepsilon}{2}S(x)^{2} +(3\gamma RT_{\infty}+1)S(x)^{3}
\right\},\\[2mm]
q_{4}(x)&=&\frac{\varepsilon R}{2(\gamma-1)T_{\infty}}
+B(x)^{-2}\Gamma^{-2}\left\{-\frac{\varepsilon
R}{2(\gamma-1)T_{\infty}}S(x)^{2} +\frac{\gamma
R}{(\gamma-1)T_{\infty}}S(x)^{3} \right\},
\end{eqnarray*}
and
\begin{equation*}
q_{5}(x)=-\varepsilon R+2\varepsilon R
B(x)^{-2}\Gamma^{-2}S(x)^{2}.
\end{equation*}
Here, functions $B(x)$ and $S(x)$ are given by
\begin{eqnarray*}
B(x)=x+\beta^{-1},
\end{eqnarray*}
and
\begin{eqnarray*}
S(x)=(x+\beta^{-1})/(x+\Gamma^{-1}\phi_{b}^{-\frac{1}{2}}),
\end{eqnarray*}
respectively.
We claim that
\begin{eqnarray}\label{2.19}
 \begin{aligned}[b]
 q_{1}(x)>0,\ \ \  q_{3}(x)>0,\ \ \
 q_{4}(x)>0,
\end{aligned}
\end{eqnarray}
\begin{eqnarray}\label{2.19a}
 \begin{aligned}[b]
q_{2}(x)^{2}-4q_{1}(x)q_{3}(x)<0, \ \ \
q_{5}(x)^{2}-4q_{3}(x)q_{4}(x)<0,
\end{aligned}
\end{eqnarray}
and
\begin{eqnarray}\label{2.20}
 \begin{aligned}[b]
q_{1}(x)q_{5}(x)^{2}+q_{4}(x)q_{2}(x)^{2}-4q_{1}(x)q_{3}(x)q_{4}(x)
\leq -c B(x)^{-2}.
\end{aligned}
\end{eqnarray}
In fact, we observe from \eqref{b}-\eqref{d} that
$$
S(x)\geq 1,\quad B(x)^{-2}\leq \beta^{2}\leq C
\phi_{b}\leq C \delta.
$$
Using the above observation and letting $\delta>0$ be small enough, it is straightforward to prove \eqref{2.19} and \eqref{2.19a}; the details are omitted for brevity.
As for \eqref{2.20}, recalling \eqref{def.conGa},  one has
\begin{align}
&q_{1}(x)q_{5}(x)^{2}+q_{4}(x)q_{2}(x)^{2}-4q_{1}(x)q_{3}(x)q_{4}(x)\notag \\[2mm]
&\leq
\frac{\varepsilon^{2}R^{2}}{2(\gamma-1)}B(x)^{-2}\Big\{
\varepsilon(\varepsilon-1)(\varepsilon-2)-2\varepsilon(1+\gamma R
T_{\infty})\Gamma^{-2}S(x)^{2}\notag \\[2mm]
&\qquad\qquad\qquad\qquad\quad-2[(\gamma^{2}+\gamma)R
T_{\infty}+2]\Gamma^{-2}S(x)^{3} \Big\}+C\beta^{2}B(x)^{-2}
\notag \\[2mm]
&\leq \frac{\varepsilon^{2}R^{2}}{2(\gamma-1)}B(x)^{-2}\left\{
\varepsilon(\varepsilon-1)(\varepsilon-2)-12(\frac{2}{\gamma+1}\varepsilon+2)+C\beta^{2} \right\}.\label{2.20p1}
\end{align}
Recall \eqref{def.la0} for the definition of $\lambda_{0}$. Since $\la\in [4,\la_0)$ and $\varepsilon\in (0,\la]$, \eqref{2.20} follows from \eqref{2.20p1} by letting $\beta^2$ be small enough. Thus, combining \eqref{2.20} together with \eqref{2.19}  and \eqref{2.19a}, it holds that
\begin{eqnarray*}
 \begin{aligned}[b]
Q(x)\geq c B(x)^{-2}(\varphi^{2}+\psi^{2}+\zeta^{2}),
\end{aligned}
\end{eqnarray*}
 which then implies that
\begin{equation}\label{2.22}
\int_{\mathbb{R}_{+}}\beta W_{\varepsilon-1,\beta}Q(x)dx\geq c\beta\int_{\mathbb{R}_{+}}
W_{\varepsilon-1,\beta} B(x)^{-2}(\varphi^{2}+\psi^{2}+\zeta^{2})dx =c
\beta^{3}\|(\varphi,\psi,\zeta)\|_{\varepsilon-3,\beta}^{2}.
\end{equation}
Therefore, the key estimate \eqref{2.16} follows by substituting \eqref{2.22} into \eqref{2.18} and letting  $\beta^{2}\leq
C\phi_{b}\leq C\delta$ and $\mathcal
{N}_{\lambda,\beta}(M)/\beta^{3}\leq \delta$ for $\delta>0$ small enough.

For the last three terms on the right-hand side of \eqref{2.14}, it is direct to obtain
\begin{align}
&C\mathcal
{N}_{\lambda,\beta}(M)\|(\varphi,\psi,\zeta,\sigma)\|^{2}_{\varepsilon-3,\beta}+C\phi_{b}\int_{\mathbb{R}_{+}}\beta
W_{\varepsilon-1,\beta}G(x)^{-2}(\varphi^{2}+\psi^{2}+\zeta^{2}+\sigma^{2})dx\notag\\
&\quad+C(\mathcal{N}_{\lambda,\beta}(M)+\phi_{b})\int_{\mathbb{R}_{+}}
W_{\varepsilon,\beta}G(x)^{-3}(\varphi^{2}+\psi^{2}+\zeta^{2}+\sigma^{2})dx\notag \\
&\leq C\delta \beta^{3}\|(\varphi,\psi,\zeta)\|_{\varepsilon-3,\beta}^{2}+C\delta \beta^{3}[\sigma_{x}^{2}]_{x=0}\label{2.23}
\end{align}
with the help of \eqref{b}, \eqref{d},  \eqref{def.Gx} and the elliptic estimates in Lemma \ref{main.result4}.

 By substituting \eqref{2.23} and \eqref{2.16} into
 \eqref{2.14}, we have
\begin{eqnarray}\label{2.24}
 \begin{aligned}[b]
I_{1}+I_{2}\geq&
(c-C\delta)(\beta^{3}\|(\varphi,\psi,\zeta)\|_{\varepsilon-3,\beta}^{2}+\beta\|\sigma_{x}\|_{\varepsilon-1,\beta}^{2})-C\beta[\sigma^{2}+\sigma_{x}^{2}]_{x=0}
\end{aligned}
\end{eqnarray}

At the end, we estimate the only term on the right-hand side of
\eqref{2.7}. In fact, recalling \eqref{2.6} as well as \eqref{def.r0} and \eqref{def.r1x}, it holds that
\begin{eqnarray}\label{2.25}
 \begin{aligned}[b]
\int_{\mathbb{R}_{+}} W_{\varepsilon,\beta}\CN_{1}dx\leq  C\delta
\left\{\beta^{3}\|(\varphi,\psi,\zeta)\|_{\varepsilon-3,\beta}^{2}+\beta\|(\varphi_{x},\psi_{x},\zeta_{x})\|_{\varepsilon-1,\beta}^{2}\right\}+C\delta \beta^{3}[\sigma_{x}^{2}]_{x=0},
\end{aligned}
\end{eqnarray}
where we have used \eqref{1.14d}, \eqref{1.16}, \eqref{b}-\eqref{d}, $\lambda\geq 4$ and the
Cauchy-Schwarz inequality and the elliptic estimate in Lemma \ref{main.result4}.

 Substituting \eqref{2.11a}, \eqref{2.11}, \eqref{2.11b}, \eqref{2.24} and \eqref{2.25}
into \eqref{2.7}, we have
\begin{align}
&\frac{d}{dt}\int_{\mathbb{R}_{+}} W_{\varepsilon,\beta}(e^{-\tilde{\phi}}\CE_0
+\CE_1^{\rm x}+\frac{1}{2}\tilde{n}^{2}\varphi^{2})dx+c\beta^{3}\|(\varphi,\psi,\zeta)\|_{\varepsilon-3,\beta}^{2}\notag\\[2mm]
&\quad +c\beta\|(\varphi_{x},\psi_{x},\zeta_{x},\sigma_{x})\|_{\varepsilon-1,\beta}^{2}
\leq C[\sigma^{2}+\sigma_{x}^{2}]_{x=0},\label{2.26}
\end{align}
provided that $\delta>0$ is sufficiently small and $\CE_0$ and $\CE_1^{\rm x}$ are defined in \eqref{def.e0} and \eqref{def.e1x}, respectively.
Furthermore, multiplying \eqref{2.26} by $(1+\beta \tau)^{\xi}$ and integrating the resulting inequality
over $(0,t)$ give the desired estimate \eqref{2.10}.  This hence completes the proof of Lemma \ref{main.result2.9}.

\end{proof}

\begin{lemma}\label{main.resultloik0}
Under the same conditions as in Proposition \ref{ste.pro1.1}, there exist
positive constants $C$ and $\delta$ independent of $M$ such that if  conditions  \eqref{b}-\eqref{d}  and  $\theta\Gamma \phi_{b}^{1/2}\leq\beta_{1}\leq \beta\leq\Gamma \phi_{b}^{1/2}$ are satisfied, it holds for
any $t\in[0,M]$ and any $\xi\geq 0$ that
\begin{align}
&(1+\beta t)^{\xi}\|(\varphi,\psi,\zeta,\sigma,\sigma_{x})(t)\|_{\varepsilon,\beta_{1}}^{2}+\int_{0}^{t} (1+\beta
\tau)^{\xi}[\varphi^{2}+\psi^{2}+\zeta^{2}+\sigma^{2}+\sigma_{x}^{2}]_{x=0}d\tau\notag\\
\leq& C\|(\varphi_{0},\psi_{0},\zeta_{0},\sigma_{0},\sigma_{x0})\|_{\varepsilon,\beta}^{2}
+C\xi\beta\int_{0}^{t} (1+\beta
\tau)^{\xi-1}\|(\varphi,\psi,\zeta,\sigma,\sigma_{x})(\tau)\|_{\varepsilon,\beta}^{2}d\tau\notag\\&+C\int_{0}^{t} (1+\beta
\tau)^{\xi}\left[(\beta_{1}^{3}+\delta\beta^{3})\|(\varphi,\psi,\zeta)(\tau)\|_{\varepsilon-3,\beta}^{2}
+(\beta_{1}+\delta\beta^{3})\|(\varphi_{x},\psi_{x},\zeta_{x},\sigma_{x})(\tau)\|_{\varepsilon-1,\beta}^{2}\right]d\tau.\label{2.uykjui}
\end{align}
\end{lemma}

\begin{proof}
To make the fourth term of \eqref{2.1} easier to calculate and using $\tilde{v}_{x}=\frac{\tilde{n}_{x}}{\tilde{n}}$, we rewrite \eqref{2.1} as follows:
\begin{equation}
\label{2.1new}
\left(\frac{\CE_0}{\tilde{n}}\right)_t+\left(\frac{\CH_0}{\tilde{n}}\right)_x+\psi_{x}\sigma=-\tilde{v}_{x}\frac{\CH_0}{\tilde{n}}
-\frac{\CD_0}{\tilde{n}}+\frac{\CR_0}{\tilde{n}}.
\end{equation}
 By $\eqref{1.16}_{1}$, we have
\begin{equation}
\label{2.2new}
\psi_{x}\sigma=-\sigma\varphi_{t}-\tilde{u}\sigma\varphi_{x}-(\sigma\psi\varphi)_{x}+\sigma\varphi\psi_{x}+\psi\varphi\sigma_{x}
-\tilde{v}_{x}\psi\sigma.
\end{equation}
Differentiating \eqref{1.17} in $t$, we have
\begin{equation}
\label{2.3new}
\sigma_{xxt}=e^{v}\varphi_{t}+e^{-\phi}\sigma_{t}.
\end{equation}
Multiplying \eqref{2.3new} by $e^{-v}\sigma$, we have
\begin{eqnarray}\label{2.3newh}
 \begin{aligned}
-\sigma\varphi_{t}=&\left[\frac{1}{2}e^{-\phi}e^{-v}\sigma^{2}+\frac{1}{2}e^{-v}\sigma_{x}^{2}\right]_{t}-(e^{-v}\sigma\sigma_{xt})_{x}
\\&+\frac{1}{2}e^{-\phi}e^{-v}(\varphi_{t}+\sigma_{t})\sigma^{2}+\frac{1}{2}e^{-v}\varphi_{t}\sigma_{x}^{2}-e^{-v}\varphi_{x}\sigma\sigma_{xt}
-e^{-v}\tilde{v}_{x}\sigma\sigma_{xt}.
\end{aligned}
\end{eqnarray}
To compute $-\tilde{u}\sigma\varphi_{x}$ in \eqref{2.2new}, we rewrite \eqref{1.17} into
\begin{equation}
\label{2.4new}
\sigma_{xx}=e^{\tilde{v}}(e^{\varphi}-1-\varphi)-e^{-\tilde{\phi}}(e^{-\sigma}-1+\sigma)+\varphi e^{\tilde{v}}+e^{-\tilde{\phi}}\sigma.
\end{equation}
Multiplying \eqref{2.4new} by $e^{-\tilde{v}}\tilde{u}\sigma_{x}$, we have
\begin{eqnarray}\label{2.5new}
 \begin{aligned}[b]
-\tilde{u}\sigma\varphi_{x}=&\left[\frac{1}{2}e^{-\tilde{v}}\tilde{u}\sigma_{x}^{2}-\tilde{u}\varphi\sigma-\frac{1}{2}\tilde{u}e^{-\tilde{\phi}}e^{-\tilde{v}}\sigma^{2}\right]_{x}
-\frac{1}{2}(e^{-\tilde{v}}\tilde{u})_{x}\sigma_{x}^{2}+\tilde{u}_{x}\varphi\sigma\\&-\tilde{u}(e^{\varphi}-1-\varphi)\sigma_{x}+e^{-\tilde{\phi}}e^{-\tilde{v}}\tilde{u}(e^{-\sigma}-1+\sigma)\sigma_{x}
+\frac{1}{2}(e^{-\tilde{\phi}}e^{-\tilde{v}}\tilde{u})_{x}\sigma^{2}.
\end{aligned}
\end{eqnarray}
Substituting \eqref{2.3newh}, \eqref{2.5new} and  \eqref{2.2new} into \eqref{2.1new}, we have
\begin{eqnarray}\label{2.6new}
 \begin{aligned}[b]
&\left(\frac{\CE_0}{\tilde{n}}+\frac{1}{2}e^{-\phi}e^{-v}\sigma^{2}+\frac{1}{2}e^{-v}\sigma_{x}^{2}\right)_t
\\&+\left(\frac{\CH_0}{\tilde{n}}+\frac{1}{2}e^{-\tilde{v}}\tilde{u}\sigma_{x}^{2}-u\varphi\sigma
-\frac{1}{2}\tilde{u}e^{-\tilde{\phi}}e^{-\tilde{v}}\sigma^{2}-e^{-v}\sigma\sigma_{xt}\right)_x
=\CN_2,
\end{aligned}
\end{eqnarray}
where
\begin{eqnarray}\label{2.6nelkw}
 \begin{aligned}[b]
\CN_2=&-\tilde{v}_{x}(\frac{\CH_0}{\tilde{n}}-\psi\sigma-e^{-v}\sigma\sigma_{xt})
-\frac{\CD_0}{\tilde{n}}
+\frac{1}{2}(e^{-\tilde{v}}\tilde{u})_{x}\sigma_{x}^{2}-\tilde{u}_{x}\varphi\sigma
-\frac{1}{2}(e^{-\tilde{\phi}}e^{-\tilde{v}}\tilde{u})_{x}\sigma^{2}
\\&+\frac{\CR_0}{\tilde{n}}
-\frac{1}{2}e^{-\phi}e^{-v}(\varphi_{t}+\sigma_{t})\sigma^{2}-\frac{1}{2}e^{-v}\varphi_{t}\sigma_{x}^{2}+e^{-v}\varphi_{x}\sigma\sigma_{xt}
-\sigma\varphi\psi_{x}-\psi\varphi\sigma_{x}\\&
+\tilde{u}(e^{\varphi}-1-\varphi)\sigma_{x}-e^{-\tilde{\phi}}e^{-\tilde{v}}\tilde{u}(e^{-\sigma}-1+\sigma)\sigma_{x}.
\end{aligned}
\end{eqnarray}

Then, multiplying \eqref{2.6new} by $ W_{\varepsilon,\beta_{1}}$ and
integrating the resulting equation over $\mathbb{R}_{+}$,
one deduces that
\begin{eqnarray}\label{2.7new}
\begin{aligned}[b]
&\frac{d}{dt}\int_{\mathbb{R}_{+}} W_{\varepsilon,\beta_{1}}\left(\frac{\CE_0}{\tilde{n}}+\frac{1}{2}e^{-\phi}e^{-v}\sigma^{2}+\frac{1}{2}e^{-v}\sigma_{x}^{2}\right)
dx\\&+\left[-\frac{\CH_0}{\tilde{n}}-\frac{1}{2}e^{-\tilde{v}}\tilde{u}\sigma_{x}^{2}+u\varphi\sigma
+\frac{1}{2}\tilde{u}e^{-\tilde{\phi}}e^{-\tilde{v}}\sigma^{2}+e^{-v}\sigma\sigma_{xt}\right]_{x=0}\\[2mm]
&+\varepsilon\beta_{1}\int_{\mathbb{R}_{+}} W_{\varepsilon-1,\beta_{1}}\left[-\frac{\CH_0}{\tilde{n}}+u\varphi\sigma
+\frac{1}{2}\tilde{u}e^{-\tilde{\phi}}e^{-\tilde{v}}\sigma^{2}-\frac{1}{2}e^{-\tilde{v}}\tilde{u}\sigma_{x}^{2}\right]dx
\\&+\varepsilon\beta_{1}\int_{\mathbb{R}_{+}} W_{\varepsilon-1,\beta_{1}}e^{-v}\sigma\sigma_{xt}dx
=\int_{\mathbb{R}_{+}} W_{\varepsilon,\beta_{1}}\CN_2 dx.
\end{aligned}
\end{eqnarray}

Due to the boundary condition \eqref{1.19}, we have
\begin{align}
 &\Big[-\frac{\CH_0}{\tilde{n}}-\frac{1}{2}e^{-\tilde{v}}\tilde{u}\sigma_{x}^{2}+u\varphi\sigma
+\frac{1}{2}\tilde{u}e^{-\tilde{\phi}}e^{-\tilde{v}}\sigma^{2}+e^{-v}\sigma\sigma_{xt}\Big]_{x=0}\notag\\[2mm]
=&\Big[-\frac{R}{2}Tu\varphi^{2}-RT\varphi\psi-\frac{1}{2}mu\psi^{2}-R\zeta\psi-\frac{R
u}{2(\gamma-1)T}\zeta^{2}\notag\\[2mm]
&+u\varphi\sigma
-\tilde{u}(1-e^{-\varphi})\sigma
+\frac{1}{2}\tilde{u}e^{-\tilde{\phi}}e^{-\tilde{v}}\sigma^{2}+u_{\infty}e^{-v}(e^{-\sigma}-1)\sigma-\frac{1}{2}e^{-\tilde{v}}\tilde{u}\sigma_{x}^{2}\Big]_{x=0}
\notag\\[2mm]\geq&\Big[-\frac{R}{2}T_{\infty}u_{\infty}\varphi^{2}-RT_{\infty}\varphi\psi-\frac{mu_{\infty}}{2}\psi^{2}-R\zeta\psi-\frac{R
u_{\infty}}{2(\gamma-1)T_{\infty}}\zeta^{2}-\frac{u_{\infty}}{2}(\sigma^{2}+\sigma_{x}^{2})\Big]_{x=0}\notag\\[2mm]&
-C(\mathcal{N}_{\lambda,\beta}(M)+\phi_{b})[\sigma^{2}+\varphi^{2}+\psi^{2}+\zeta^{2}]_{x=0}\notag\\[2mm]
\geq& (c-C\delta)[\varphi^{2}+\psi^{2}+\zeta^{2}+\sigma^{2}+\sigma_{x}^{2}]_{x=0}
,\label{2.8new}
\end{align}
where the last inequality's computations are similar to obtain  \eqref{2.10a}.

Using \eqref{2.3} and $\beta_{1}\leq \beta$, the estimate for the
$\varepsilon\beta_{1}\int_{\mathbb{R}_{+}} W_{\varepsilon-1,\beta_{1}}\left[u\varphi\sigma
+\frac{1}{2}\tilde{u}e^{-\tilde{\phi}}e^{-\tilde{v}}\sigma^{2}\right]dx$ is given by
\begin{eqnarray}\label{2.9new}
\begin{aligned}[b]
&\varepsilon\beta_{1}\int_{\mathbb{R}_{+}} W_{\varepsilon-1,\beta_{1}}\left[u\varphi\sigma
+\frac{1}{2}\tilde{u}e^{-\tilde{\phi}}e^{-\tilde{v}}\sigma^{2}\right]dx\\
=&\varepsilon\beta_{1}\int_{\mathbb{R}_{+}} W_{\varepsilon-1,\beta_{1}}\left[\tilde{u}\varphi\sigma
+\frac{1}{2}\tilde{u}e^{-\tilde{\phi}}e^{-\tilde{v}}\sigma^{2}+\psi\varphi\sigma\right]dx\\
\geq&\varepsilon\beta_{1}\int_{\mathbb{R}_{+}} W_{\varepsilon-1,\beta_{1}}\left[\tilde{u}e^{-\tilde{v}}\sigma_{xx}\sigma
-\frac{1}{2}\tilde{u}e^{-\tilde{\phi}}e^{-\tilde{v}}\sigma^{2}\right]dx-C\beta_{1}\mathcal{N}_{\lambda,\beta}(M)\|(\varphi,\sigma)\|_{\varepsilon-3,\beta_{1}}^{2}
\\
\geq&\varepsilon\beta_{1}\int_{\mathbb{R}_{+}} W_{\varepsilon-1,\beta_{1}}\left[-\tilde{u}e^{-\tilde{v}}\sigma_{x}^{2}
-\frac{1}{2}\tilde{u}e^{-\tilde{\phi}}e^{-\tilde{v}}\sigma^{2}\right]dx-\varepsilon\beta_{1}\int_{\mathbb{R}_{+}} W_{\varepsilon-1,\beta_{1}}(\tilde{u}e^{-\tilde{v}})_{x}\sigma\sigma_{x}dx\\&-\varepsilon(\varepsilon-1)\beta_{1}^{2}\int_{\mathbb{R}_{+}} W_{\varepsilon-2,\beta_{1}}\tilde{u}e^{-\tilde{v}}\sigma\sigma_{x}dx
-C\beta_{1}\mathcal{N}_{\lambda,\beta}(M)\|(\varphi,\sigma)\|_{\varepsilon-3,\beta_{1}}^{2}
-C\beta_{1}[\sigma^{2}+\sigma_{x}^{2}]_{x=0}
\\
\geq&\varepsilon\beta_{1}\int_{\mathbb{R}_{+}} W_{\varepsilon-1,\beta_{1}}\left[-\tilde{u}e^{-\tilde{v}}\sigma_{x}^{2}
-\frac{1}{2}\tilde{u}e^{-\tilde{\phi}}e^{-\tilde{v}}\sigma^{2}\right]dx
-C\beta_{1}(\beta_{1}^{2}+\delta\beta^{3})\|(\varphi,\sigma)\|_{\varepsilon-3,\beta_{1}}^{2}\\&-C\beta_{1}\|\sigma_{x}\|_{\varepsilon-1,\beta_{1}}^{2}
-C\beta_{1}[\sigma^{2}+\sigma_{x}^{2}]_{x=0}
\end{aligned}
\end{eqnarray}
where the boundedness of $S(x)$ are used in deriving the last inequality. First, one can decompose $v$, $u$ and $T$ as $v=\varphi+(\tilde{v}-0)$, $u=\psi+(\tilde{u}-u_{\infty})+u_{\infty}$ and  $T=\zeta+(\tilde{T}-T_{\infty})+T_{\infty}$, respectively. Then using \eqref{2.9new}, integration by parts, the Schwarz inequality, \eqref{1.14d}, \eqref{1.15b}, Lemma \ref{main.result4.2} and  Lemma \ref{main.result4}, we have the following estimates:
\begin{align}
&\varepsilon\beta_{1}\int_{\mathbb{R}_{+}} W_{\varepsilon-1,\beta_{1}}\Big[-\frac{\CH_0}{\tilde{n}}+u\varphi\sigma
+\frac{1}{2}\tilde{u}e^{-\tilde{\phi}}e^{-\tilde{v}}\sigma^{2}-\frac{1}{2}e^{-\tilde{v}}\tilde{u}\sigma_{x}^{2}\Big]dx\notag\\[2mm]
\geq&\varepsilon\beta_{1}\int_{\mathbb{R}_{+}} W_{\varepsilon-1,\beta_{1}}\Big[\frac{1}{2}RT_{\infty}(-u_{\infty})\varphi^{2}-RT_{\infty}\varphi\psi+\frac{m}{2}(-u_{\infty})\psi^{2}
-R\zeta\psi+\frac{R
(-u_{\infty})}{2(\gamma-1)T_{\infty}}\zeta^{2}\notag\\[2mm]
&+\sigma\psi+\frac{-u_{\infty}}{2}\sigma^{2}+\frac{3(-u_{\infty})}{2}\sigma_{x}^{2}\Big]dx
-C\beta_{1}(\beta_{1}^{2}+\mathcal{N}_{\lambda,\beta}(M))\|(\varphi,\psi,\zeta,\sigma)\|_{\varepsilon-3,\beta_{1}}^{2}\notag\\[2mm]&-C\beta_{1}\|\sigma_{x}\|_{\varepsilon-1,\beta_{1}}^{2}
-C\beta_{1}[\sigma^{2}+\sigma_{x}^{2}]_{x=0}\notag\\[2mm]
\geq&\varepsilon\beta_{1}\int_{\mathbb{R}_{+}} W_{\varepsilon-1,\beta_{1}}\Big[\frac{RT_{\infty}}{2}|u_{\infty}|(\varphi-\frac{1}{|u_{\infty}|}\psi)^{2}
+\frac{(\gamma-1)RT_{\infty}}{2|u_{\infty}|}(\psi-\frac{|u_{\infty}|}{(\gamma-1)T_{\infty}}\zeta)^{2}\notag\\[2mm]&+\frac{|u_{\infty}|}{2}(\sigma+\frac{1}{|u_{\infty}|}\psi)^{2}\Big]dx-C\beta_{1}(\beta_{1}^{2}
+\mathcal{N}_{\lambda,\beta}(M))\|(\varphi,\psi,\zeta,\sigma)\|_{\varepsilon-3,\beta_{1}}^{2}\notag\\[2mm]&-C\beta_{1}\|\sigma_{x}\|_{\varepsilon-1,\beta_{1}}^{2}
-C\beta_{1}[\sigma^{2}+\sigma_{x}^{2}]_{x=0}\notag\\[2mm]
\geq&-C\beta_{1}(\beta_{1}^{2}+\delta\beta^{3})\|(\varphi,\psi,\zeta)\|_{\varepsilon-3,\beta_{1}}^{2}-C\beta_{1}\|\sigma_{x}\|_{\varepsilon-1,\beta_{1}}^{2}
-C\beta_{1}[\sigma^{2}+\sigma_{x}^{2}]_{x=0},
\label{2.10new}
\end{align}
\begin{eqnarray*}\label{2.11new}
\begin{aligned}[b]
&\varepsilon\beta_{1}\int_{\mathbb{R}_{+}} W_{\varepsilon-1,\beta_{1}}e^{-v}\sigma\sigma_{xt}dx\\
=&\varepsilon\beta_{1}\int_{\mathbb{R}_{+}} (W_{\varepsilon-1,\beta_{1}}e^{-v}\sigma\sigma_{t})_{x}dx
-\varepsilon\beta_{1}\int_{\mathbb{R}_{+}} W_{\varepsilon-1,\beta_{1}}e^{-v}\sigma_{x}\sigma_{t}dx\\&
-\varepsilon(\varepsilon-1)\beta_{1}^{2}\int_{\mathbb{R}_{+}} W_{\varepsilon-2,\beta_{1}}e^{-v}\sigma\sigma_{t}dx
+\varepsilon\beta_{1}\int_{\mathbb{R}_{+}} W_{\varepsilon-1,\beta_{1}}e^{-v}v_{x}\sigma\sigma_{t}dx,
\end{aligned}
\end{eqnarray*}
\begin{eqnarray}\label{2.12new}
\begin{aligned}[b]
&\left|\varepsilon\beta_{1}\int_{\mathbb{R}_{+}} W_{\varepsilon-1,\beta_{1}}e^{-v}\sigma\sigma_{xt}dx\right|\\
\leq&C\beta_{1}[\sigma^{2}+\sigma_{t}^{2}]_{x=0}+C\beta_{1}\|(\sigma_{x},\sigma_{t})\|_{\varepsilon-1,\beta_{1}}^{2}
+C\beta_{1}^{3}\|\sigma\|_{\varepsilon-3,\beta_{1}}^{2}\\&+C\beta_{1}\|\sigma\|_{\infty}\|(\sigma_{x},\varphi_{x})\|_{\varepsilon-1,\beta_{1}}^{2}
\\
\leq&C\beta_{1}[\varphi^{2}+\psi^{2}+\sigma^{2}]_{x=0}+C\beta_{1}\|(\sigma_{x},\varphi_{t})\|_{\varepsilon-1,\beta_{1}}^{2}
+C\beta_{1}^{3}\|\sigma\|_{\varepsilon-3,\beta_{1}}^{2}\\&+C\mathcal{N}_{\lambda,\beta}(M)\beta_{1}\|(\sigma_{x},\varphi_{x})\|_{\varepsilon-1,\beta_{1}}^{2}
\\
\leq&C\beta_{1}[\varphi^{2}+\psi^{2}+\sigma^{2}+\sigma_{x}^{2}]_{x=0}+C(\beta_{1}+\mathcal{N}_{\lambda,\beta}(M)\beta_{1})\|(\sigma_{x},\psi_{x},\varphi_{x})\|_{\varepsilon-1,\beta_{1}}^{2}
\\&+C\beta_{1}^{3}\|(\varphi,\psi)\|_{\varepsilon-3,\beta_{1}}^{2}
\end{aligned}
\end{eqnarray}

Notice that $\beta_{1}\leq \beta$  and the last term in \eqref{2.7new} is estimated as
\begin{eqnarray}\label{2.13new}
\begin{aligned}[b]
\int_{\mathbb{R}_{+}} W_{\varepsilon,\beta_{1}}\CN_2 dx\leq & C(\beta_{1}^{3}+\mathcal{N}_{\lambda,\beta}(M))\|(\varphi,\psi,\zeta,\sigma)\|_{\varepsilon-3,\beta_{1}}^{2}
\\&+C(\mathcal{N}_{\lambda,\beta}(M)+\beta_{1}^{2})\|(\varphi_{x},\psi_{x},\zeta_{x},\sigma_{x},\varphi_{t},\zeta_{t},\sigma_{t},\sigma_{tx})\|_{\varepsilon-1,\beta_{1}}^{2}
\\ \leq &C(\beta_{1}^{3}+\delta\beta^{3})\|(\varphi,\psi,\zeta)\|_{\varepsilon-3,\beta_{1}}^{2}+C(\beta_{1}^{2}+\delta\beta^{3})\|(\varphi_{x},\psi_{x},\zeta_{x},\sigma_{x})\|_{\varepsilon-1,\beta_{1}}^{2} \\&+C(\beta_{1}^{2}+\delta\beta^{3})[\varphi^{2}+\psi^{2}+\sigma^{2}+\sigma_{x}^{2}]_{x=0},
\end{aligned}
\end{eqnarray}
where we repeatedly use  Lemma \ref{main.result4.2} and  Lemma \ref{main.result4}.

Substituting \eqref{2.8new}-\eqref{2.13new} into \eqref{2.7new}, we have
\begin{align}
&\frac{d}{dt}\int_{\mathbb{R}_{+}} W_{\varepsilon,\beta_{1}}\left(\frac{\CE_0}{\tilde{n}}+\frac{1}{2}e^{-\phi}e^{-v}\sigma^{2}+\frac{1}{2}e^{-v}\sigma_{x}^{2}\right)
dx+c[\varphi^{2}+\psi^{2}+\zeta^{2}+\sigma^{2}+\sigma_{x}^{2}]_{x=0}\notag\\
\leq &C(\beta_{1}^{3}+\delta\beta^{3})\|(\varphi,\psi,\zeta)\|_{\varepsilon-3,\beta}^{2}+C(\beta_{1}+\delta\beta^{3})\|(\varphi_{x},\psi_{x},\zeta_{x},\sigma_{x})\|_{\varepsilon-1,\beta}^{2},\label{2.14new}
\end{align}
provided that $0<\beta_{1}\leq \beta$  and $\delta$ are sufficiently small, where $\CE_0$ is defined in \eqref{def.e0}. Furthermore, multiplying \eqref{2.14new} by $(1+\beta \tau)^{\xi}$ and integrating the resulting inequality
over $(0,t)$ give the desired estimate \eqref{2.uykjui}.  This  completes the proof of Lemma \ref{main.resultloik0}.
\end{proof}

In Lemmata \ref{main.result2.9} and \ref{main.resultloik0}, we have obtained an a priori estimate up to the first order derivatives.
Following almost the same arguments, we give estimates of the highest order derivatives in the following Lemmata \ref{main.result2.erer9} and \ref{mainboundary2g}.
The computations in Lemma \ref{main.result2.erer9} are only formal in the sense
that they contain the derivatives with order exceeding two. For the complete argument, we have to operate mollifiers
in the temporal direction to cover the lack of regularity and tend the size of mollifiers
to zero after integration in time and space to deduce the desired estimates. However,
we omit such standard arguments and proceed assuming that the regularity of the
solution is as high as we need in the formal computations.
Obviously, Lemmata \ref{main.result2.erer9} and \ref{mainboundary2g} correspond to Lemmata \ref{main.result2.9} and \ref{main.resultloik0}, respectively.
Hence, the proofs are given just briefly.

\begin{lemma}\label{main.result2.erer9}
Under the same conditions as in Proposition \ref{ste.pro1.1}, there exist
positive constants $C$ and $\delta$ independent of $M$ such that if
conditions \eqref{b}-\eqref{d} are satisfied, it holds for
any $t\in[0,M]$ and any $\xi\geq 0$ that
\begin{align}
 &(1+\beta t)^{\xi}\|(\varphi_{t},\psi_{t},\zeta_{t})(t)\|_{\varepsilon,\beta,1}^{2}\notag\\
&\quad +\int_{0}^{t} (1+\beta
 \tau)^{\xi}\left\{\beta^{3}\|(\varphi_{t},\psi_{t},\zeta_{t})\|_{\varepsilon-3,\beta}^{2}
 +\beta\|(\varphi_{tx},\psi_{tx},\zeta_{tx},\sigma_{tx})\|_{\varepsilon-1,\beta}^{2}\right\}d\tau\notag\\
& \leq C\|(\varphi_{0t},\psi_{0t},\zeta_{0t})\|_{\varepsilon,\beta,1}^{2}
 +C\xi\beta\int_{0}^{t} (1+\beta
 \tau)^{\xi-1}\|(\varphi_{t},\psi_{t},\zeta_{t})\|_{\varepsilon,\beta,1}^{2}
 d\tau\notag\\
 &\quad+C\int_{0}^{t}(1+\beta
\tau)^{\xi}\left[
\delta\beta\|(\varphi_{x},\psi_{x},\zeta_{x})\|_{\varepsilon-1,\beta}^{2}+\delta\beta^{3}\|(\varphi,\psi,\zeta,\varphi_{xx},\psi_{xx},\zeta_{xx})\|_{\varepsilon-3,\beta}^{2}\right]d\tau
\notag\\&\quad+C\int_{0}^{t}(1+\beta
\tau)^{\xi}[\varphi^{2}+\psi^{2}+\sigma^{2}+\sigma_{t}^{2}]_{x=0}d\tau.
\label{2.35}
 \end{align}
\end{lemma}

\begin{proof}
We follow the same steps as in deriving \eqref{2.10} in the proof of Lemma \ref{main.result2.9}.
On one hand, taking the time derivative on \eqref{2.0}, then taking the inner product of the resulting system with $\tilde{n}(\varphi_{t},\psi_{t},\zeta_{t})$ and using $\tilde{v}_{x}=\frac{\tilde{n}_{x}}{\tilde{n}}$, it follows that
\begin{equation}\label{2.27}
(\CE_1^{\rm t})_t+(\CH_1^{\rm t})_{x}+\CD_1^{\rm t}+\tilde{n}\psi_{xt}\sigma_{t}=\CR_1^{\rm t},
\end{equation}
where we have denoted
\begin{equation}
\label{1}
\CE_1^{\rm t}=\frac{\tilde{n}}{2}RT\varphi_{t}^{2}+\frac{\tilde{n}}{2}m\psi_{t}^{2}+\frac{\tilde{n}R}{2(\gamma-1)T}\zeta_{t}^{2},
\end{equation}
\begin{equation}
\label{2}
\CH_1^{\rm t}=\frac{\tilde{n}}{2}RTu\varphi_{t}^{2}+\tilde{n}RT\varphi_{t}\psi_{t}+\frac{\tilde{n}}{2}mu\psi_{t}^{2}+R\tilde{n}\zeta_{t}\psi_{t}+\frac{\tilde{n}R
u}{2(\gamma-1)T}\zeta_{t}^{2}
-\tilde{n}\sigma_{t}\psi_{t},
\end{equation}
\begin{eqnarray}
\CD_1^{\rm t}&=&\left(-\frac{R T
u}{2}\tilde{n}_{x}-\frac{R \tilde{n}u}{2}\tilde{T}_{x}-\frac{R
\tilde{n}T}{2}\tilde{u}_{x}\right)\varphi_{t}^{2}-\tilde{n}R\tilde{T}_{x}\varphi_{t}\psi_{t}\notag\\[2mm]
&&+\left(\frac{m\tilde{n}}{2}\tilde{u}_{x}-\frac{mu}{2}\tilde{n}_{x}\right)\psi_{t}^{2}
+\frac{R\tilde{n}}{(\gamma-1)T}\tilde{T}_{x}\zeta_{t}\psi_{t}+\tilde{n}_{x}\sigma_{t}\psi_{t}\notag\\[2mm]
&&+\left[\frac{R\tilde{n}}{T}\tilde{u}_{x}
-\frac{Ru\tilde{n}_{x}+R\tilde{n}\tilde{u}_{x}}{2(\gamma-1)T}+\frac{R
u\tilde{n} \tilde{T}_{x}}{2(\gamma-1)T^{2}}\right]\zeta_{t}^{2},\label{3}
\end{eqnarray}
and
\begin{align}
\CR_1^{\rm t}&=\left(-\frac{\tilde{n}R}{2}\zeta_{t}+\frac{\tilde{n}R
u}{2}\zeta_{x}+\frac{\tilde{n}R T}{2}\psi_{x}\right)\varphi_{t}^{2}
+R\tilde{n}\zeta_{x}\varphi_{t}\psi_{t}-\frac{m\tilde{n}}{2}\psi_{x}\psi_{t}^{2}\notag\\[3mm]
&\quad-(\tilde{n}Ru\zeta_{t}+\tilde{n}R
T\psi_{t})\varphi_{x}\varphi_{t}-\frac{R \tilde{n} \psi_{t}\zeta_{x}\zeta_{t}}{(\gamma-1)T}\notag\\[3mm]
&\quad
+R\tilde{n}\left(\frac{\psi_{x}}{2(\gamma-1)T}+\frac{\zeta_{t}+u\zeta_{x}}{2(\gamma-1)T^{2}}
+\frac{\tilde{u}_{x}\zeta}{T^{2}}
+\frac{\tilde{T}_{x}\psi}{(\gamma-1)T^{2}}\right)\zeta_{t}^{2}\notag\\[3mm]
&\quad-\tilde{n}R\zeta_{t}(\psi_{x}\varphi_{t}+\psi_{t}\varphi_{x})-R\tilde{n}_{x}\psi\zeta_{t}\varphi_{t}.\label{4}
\end{align}
On the other hand, taking derivatives with respect to $x$ and $t$ about \eqref{2.0}, then taking the inner product of the resulting system with $\tilde{n}(\varphi_{xt},\psi_{xt},\zeta_{xt})$ and using $\tilde{v}_{x}=\frac{\tilde{n}_{x}}{\tilde{n}}$ again, it also follows that
\begin{equation}\label{2.28}
(\CE_2^{\rm xt})_t+(\CH_2^{\rm xt})_x
-\tilde{n}\psi_{xt}\sigma_{xxt}=\CR_2^{\rm xt},
\end{equation}
where we have denoted
\begin{equation*}
\CE_2^{\rm xt}=\frac{\tilde{n}}{2}RT\varphi_{xt}^{2}+\frac{\tilde{n}}{2}m\psi_{xt}^{2}+\frac{\tilde{n}R}{2(\gamma-1)T}\zeta_{xt}^{2},
\end{equation*}
\begin{equation*}
\CH_2^{\rm xt}=\frac{\tilde{n}}{2}RTu\varphi_{xt}^{2}+\tilde{n}RT\varphi_{xt}\psi_{xt}+\frac{\tilde{n}}{2}mu\psi_{xt}^{2}+R\tilde{n}\zeta_{xt}\psi_{xt}+\frac{\tilde{n}R
u}{2(\gamma-1)T}\zeta_{xt}^{2},
\end{equation*}
and
\begin{equation*}
\CR_2^{\rm xt}=\CR_{2,1}^{\rm xt}+\CR_{2,2}^{\rm xt}
\end{equation*}
with
\begin{align}
\CR_{2,1}^{\rm xt}
&=
\left[-\frac{\tilde{n}R}{2}\zeta_{t}+ \frac{R T
u}{2}\tilde{n}_{x}-\frac{R \tilde{n}}{2}(T
u)_{x}\right]\varphi_{xt}^{2} +\left[\frac{
mu}{2}\tilde{n}_{x}-\frac{m\tilde{n}}{2}u_{x}-m\tilde{n} \tilde{u}
_{x} \right]\psi_{xt}^{2}\notag \\[3mm]
&\quad+\left[ \frac{R
u}{2(\gamma-1)T}\tilde{n}_{x}+\frac{R
\tilde{n}u}{2(\gamma-1)T^{2}}T_{x}+\frac{\tilde{n}R}{2(\gamma-1)T^{2}}\zeta_{t}-\frac{R
\tilde{n}}{2(\gamma-1)T}u_{x}\right]\zeta_{xt}^{2}\notag\\[3mm]
&\quad
 -R\tilde{n}T_{x}(\varphi_{tt}+\psi
_{xt})\varphi_{xt}+ RT\tilde{n}_{x}\psi
_{xt}\varphi_{xt}+\frac{R\tilde{n}T_{x}}{(\gamma-1)T^{2}}\zeta_{tt}\zeta_{xt}\notag\\[3mm]
&\quad-\tilde{n}\tilde{u}_{xx}\left[m\psi_{t}\psi_{xt}+\left(\frac{R\zeta}{T}\right)_{t}\zeta_{xt}\right]
-R\tilde{n}\varphi_{xt}\varphi_{t}\zeta_{xt}
-\frac{R}{\gamma-1}\left(\frac{\tilde{n}}{T}\right)_{xt}\zeta_{xt}\zeta_{t}\notag
\end{align}
and
\begin{align*}
\CR_{2,2}^{\rm xt}&=
-\tilde{n}(RTu)_{xt}\varphi_{xt}\varphi_{x}
-\tilde{n}m\psi_{x}\psi_{xt}^{2}-\frac{R\tilde{n}}{\gamma-1}\left(\frac{u}{T}\right)_{xt}\zeta_{xt}\zeta_{x}
-R\tilde{n}(\varphi_{xt}\psi_{x}+\psi_{xt}\varphi_{x})\zeta_{xt}\notag\\[2mm]
&\quad-\tilde{n}(RTu)_{t}\varphi_{xt}\varphi_{xx}
-R\tilde{n}(\varphi_{xt}\psi_{xx}+\psi_{xt}\varphi_{xx})\zeta_{t}-\tilde{n}m\psi_{t}\psi_{xt}\psi_{xx}\notag\\[2mm]
&\quad
-\frac{R\tilde{n}}{\gamma-1}\left(\frac{u}{T}\right)_{t}\zeta_{xt}\zeta_{xx}-\tilde{n}_{x}(RT\psi)_{xt}\varphi_{xt}
-\tilde{n}\tilde{u}_{x}\left(\frac{R\zeta}{T}\right)_{xt}\zeta_{xt}-\frac{R\tilde{n}}{\gamma-1}\tilde{T}_{x}
\left(\frac{\psi}{T}\right)_{xt}\zeta_{xt}\notag\\[2mm]
&\quad
-\tilde{n}\tilde{v}_{xx}\left[(RT\psi)_{t}\varphi_{xt}+R\zeta_{t}\psi_{xt}\right]-\frac{R\tilde{n}}{\gamma-1}\tilde{T}_{xx}\left(\frac{\psi}{T}\right)_{t}\zeta_{xt}.
\end{align*}
Then, with the help of Poisson equation \eqref{1.17}, multiplying \eqref{2.27} by $e^{-\tilde{\phi}}$, adding the
resulting equation to \eqref{2.28} and dealing with the term $\tilde{n}\psi_{xt}\sigma_{t}\cdot e^{-\tilde{\phi}}-\tilde{n}\psi_{xt}\sigma_{xxt}$
as the same method as \eqref{2.4} give that
\begin{eqnarray}
&\dis \left(e^{-\tilde{\phi}}\CE_1^{\rm t}+\CE_2^{\rm xt}+\frac{1}{2}\tilde{n}^{2}\varphi_{t}^{2}\right)_{t}
+\left(e^{-\tilde{\phi}}\CH_1^{\rm t}+\CH_2^{\rm xt}
+\frac{1}{2}\tilde{n}^{2}u\varphi_{t}^{2}\right)_{x}
+e^{-\tilde{\phi}}\tilde{\phi}_{x}\CH_1^{\rm t}+e^{-\tilde{\phi}}\CD_1^{\rm t}\notag \\[2mm]
&\dis +\left(-\tilde{n}u\tilde{n}_{x}-\frac{1}{2}\tilde{n}^{2}\tilde{u}_{x}
\right)\varphi_{t}^{2}
+\tilde{n}^{2}\tilde{v}_{x}\varphi_{t}\psi_{t}=\CN_{3},
\label{2.31}
\end{eqnarray}
where
\begin{equation*}
\CN_3 =e^{-\tilde{\phi}}\CR_1^{\rm t}+\CR_2^{\rm xt}
+\frac{1}{2}\tilde{n}^{2}\psi_{x}\varphi_{t}^{2}
+\frac{\tilde{n}}{2}\left(\tilde{n}e^{\theta_{1}\varphi}\varphi^{2}
-e^{-(\theta_{2}\sigma+\tilde{\phi})}\sigma^{2}\right)_{t}\psi_{xt}
-\tilde{n}^{2}\psi_{t}\varphi_{x}\varphi_{t}.
\end{equation*}
Similarly for deriving \eqref{2.7}, it follows from \eqref{2.31} that
\begin{align}
&\frac{d}{dt}\int_{\mathbb{R}_{+}} W_{\varepsilon,\beta}
\left[e^{-\tilde{\phi}}\CE_1^{\rm t}+\CE_2^{\rm xt}+\frac{1}{2}\tilde{n}^{2}\varphi_{t}^{2}\right]dx\notag\\[3mm]
&\quad+\underbrace{\int_{\mathbb{R}_{+}}\varepsilon\beta W_{\varepsilon-1,\beta}\CL_3dx}_{I_{3}}
+\int_{\mathbb{R}_{+}}\varepsilon\beta
W_{\varepsilon-1,\beta} (-\CH_2^{\rm xt})
dx+\underbrace{\int_{\mathbb{R}_{+}} W_{\varepsilon,\beta}\CL_{4} dx}_{I_{4}}\notag\\[3mm]
&\quad
-\left[e^{-\tilde{\phi}}\CH_1^{\rm t}+\CH_2^{\rm xt}
+\frac{1}{2}\tilde{n}^{2}u\varphi_{t}^{2}\right](t,0)
=\int_{\mathbb{R}_{+}} W_{\varepsilon,\beta}\CN_{3} dx,
\label{2.33}
\end{align}
where we have denoted
\begin{equation*}
\CL_3=
-e^{-\tilde{\phi}}\CH_1^{\rm t}-\frac{1}{2}\tilde{n}^{2}u\varphi_{t}^{2}
\end{equation*}
and
 \begin{align*}
\CL_{4}=&e^{-\tilde{\phi}}\tilde{\phi}_x\CH_1^t +e^{-\tilde{\phi}}\CD_1^t+\left(-\tilde{n}u\tilde{n}_{x}-\frac{1}{2}\tilde{n}^{2}\tilde{u}_{x}
\right)\varphi_{t}^{2}
+\tilde{n}^{2}\tilde{v}_{x}\varphi_{t}\psi_{t}.
\end{align*}

Due to the similar computations as \eqref{2.11a}-\eqref{2.11b},\eqref{2.24} and \eqref{2.25}, we have the following estimates
\begin{equation*}
\int_{\mathbb{R}_{+}}\varepsilon\beta
W_{\varepsilon-1,\beta} (-\CH_2^{\rm xt})
dx\geq
c\beta\|(\varphi_{xt},\psi_{xt},\zeta_{xt})\|_{\varepsilon-1,\beta}^{2},
\end{equation*}
\begin{equation*}
-\left[e^{-\tilde{\phi}}\CH_1^{\rm t}
+\frac{1}{2}\tilde{n}^{2}u\varphi_{t}^{2}\right](t,0)\geq (c-C\delta)[\varphi_{t}^{2}+\psi_{t}^{2}+\zeta_{t}^{2}]_{x=0}-C[\sigma_{t}^{2}]_{x=0},
\end{equation*}
\begin{equation*}
 -\CH_2^{\rm xt}(t,0)\geq (c-C\delta)[\varphi_{xt}^{2}+\psi_{xt}^{2}+\zeta_{xt}^{2}]_{x=0},
\end{equation*}
\begin{equation*}
I_{3}+I_{4}\geq
(c-C\delta)\left\{\beta^{3}\|(\varphi_{t},\psi_{t},\zeta_{t})\|_{\varepsilon-3,\beta}^{2}+\beta\|\sigma_{xt}\|_{\varepsilon-1,\beta}^{2}
\right\}-C\beta[\sigma_{t}^{2}+\varphi^{2}+\psi^{2}+\sigma^{2}]_{x=0},
\end{equation*}
and
\begin{eqnarray*}
\int_{\mathbb{R}_{+}} W_{\varepsilon,\beta}\CN_{3} dx & \leq & C\delta
\beta\|(\varphi_{x},\psi_{x},\zeta_{x},\varphi_{tx},\psi_{tx},\zeta_{tx})\|_{\varepsilon-1,\beta}^{2}
\notag\\
&&+C\delta\beta^{3}\|(\varphi,\psi,\zeta,\varphi_{t},\psi_{t},\zeta_{t},\varphi_{xx},\psi_{xx},\zeta_{xx})\|_{\varepsilon-3,\beta}^{2}
\notag\\
&&+C\delta
\beta[\varphi^{2}+\psi^{2}+\sigma^{2}]_{x=0},
\end{eqnarray*}
where we make full use of \eqref{1.19}, \eqref{1.14d}, \eqref{1.16}, \eqref{b}-\eqref{d}, $\lambda\geq 4$, the Cauchy-Schwarz inequality and Lemmata \ref{main.result4} and \ref{main.result4.2}. At last, substituting the above estimates into \eqref{2.33}, the desired estimate \eqref{2.35} follows by multiplying the resulting inequality  by $(1+\beta \tau)^{\xi}$ and integrating the resulting inequality
over $(0,t)$. Therefore, we complete the rough proof of Lemma \ref{main.result2.erer9}.
\end{proof}

\begin{lemma}\label{mainboundary2g}
Under the same conditions as in Proposition \ref{ste.pro1.1}, there exist
positive constants $C$ and $\delta$ independent of $M$ such that if
conditions \eqref{b}-\eqref{d}  and $\theta\Gamma \phi_{b}^{1/2}\leq\beta_{1}\leq \beta\leq\Gamma \phi_{b}^{1/2}$ are satisfied , it holds for
any $t\in[0,M]$ and any $\xi\geq 0$ that
\begin{align}
&(1+\beta t)^{\xi}\|(\varphi_{t},\psi_{t},\zeta_{t},\sigma_{t},\sigma_{tx})(t)\|_{\varepsilon,\beta_{1}}^{2}+\int_{0}^{t} (1+\beta
\tau)^{\xi}\left[\varphi_{t}^{2}+\psi_{t}^{2}+\zeta_{t}^{2}+\sigma_{t}^{2}+\sigma_{xt}^{2}\right]_{x=0}d\tau\notag\\
\leq& C\|(\varphi_{t0},\psi_{t0},\zeta_{t0},\sigma_{t0},\sigma_{tx0})\|_{\varepsilon,\beta}^{2}
+C\xi\beta\int_{0}^{t} (1+\beta
\tau)^{\xi-1}\|(\varphi_{t},\psi_{t},\zeta_{t},\sigma_{t},\sigma_{tx})(\tau)\|_{\varepsilon,\beta}^{2}d\tau\notag\\&
+C(\beta_{1}+\delta\beta^{3})\int_{0}^{t} (1+\beta
\tau)^{\xi}\|(\varphi_{x},\psi_{x},\zeta_{x},\sigma_{x},\sigma_{xt},\varphi_{xt},\psi_{xt})\|_{\varepsilon-1,\beta}^{2}d\tau
\notag\\&
+C(\beta_{1}^{3}+\delta\beta^{3})\int_{0}^{t} (1+\beta
\tau)^{\xi}\|(\varphi,\psi,\zeta,\varphi_{t},\psi_{t},\zeta_{t})\|_{\varepsilon-3,\beta}^{2}d\tau
\notag\\&
+C(\beta_{1}+\delta\beta^{3})\int_{0}^{t} (1+\beta
\tau)^{\xi}[\sigma^{2}+\varphi^{2}+\psi^{2}]_{x=0}d\tau.
\label{2.uhgkyi}
\end{align}

\end{lemma}

\begin{proof}
We follow the similar steps as in deriving \eqref{2.6new} in the proof of Lemma \ref{main.resultloik0}. Then after we rewrite the fourth term
$\tilde{n}\psi_{xt}\sigma_{t}$ in \eqref{2.27}, we have the following equality:
\begin{align}
&\left[\CE_1^{\rm t}+\frac{\tilde{n}}{2}e^{-\phi}e^{-v}\sigma_{t}^{2}+\frac{\tilde{n}}{2}e^{-v}\sigma_{xt}^{2}\right]_{t}\notag\\&+\left[\CH_1^{\rm t}-\tilde{n}u\varphi_{t}\sigma_{t}
-\frac{1}{2}\tilde{n}ue^{-\phi}e^{-v}\sigma_{t}^{2}+\frac{1}{2}\tilde{n}ue^{-v}\sigma_{xt}^{2}-\tilde{n}e^{-v}\sigma_{t}\sigma_{ttx}\right]_{x}=
\CN_{4},\label{2.1tnew}
\end{align}
where
\begin{align}
\CN_{4}=&\CR_1^{\rm t}-\CD_1^{\rm t}-\frac{\tilde{n}}{2}(e^{-\phi}e^{-v})_{t}\sigma_{t}^{2}-(\tilde{n}e^{-v})_{x}\sigma_{t}\sigma_{ttx}+\frac{1}{2}\tilde{n}(e^{-v})_{t}\sigma_{xt}^{2}+
\tilde{n}(e^{-v})_{t}\sigma_{t}\sigma_{txx}\notag\\&+\frac{1}{2}(\tilde{n}ue^{-v})_{x}\sigma_{xt}^{2}-(\tilde{n}u)_{x}\varphi_{t}\sigma_{t}
-\frac{1}{2}(\tilde{n}ue^{-\phi}e^{-v})_{x}\sigma_{t}^{2}+\tilde{n}\psi_{t}v_{x}\sigma_{t}.\label{2.1tnewj}
\end{align}

Then, multiplying \eqref{2.1tnew} by $ W_{\varepsilon,\beta_{1}}$ and
integrating the resulting equation over $\mathbb{R}_{+}$,
one deduces that
\begin{eqnarray}\label{2.7tbnew}
\begin{aligned}[b]
&\frac{d}{dt}\int_{\mathbb{R}_{+}} W_{\varepsilon,\beta_{1}}\left[\CE_1^{\rm t}+\frac{\tilde{n}}{2}e^{-\phi}e^{-v}\sigma_{t}^{2}+\frac{\tilde{n}}{2}e^{-v}\sigma_{xt}^{2}\right]
dx\\&+\left[-\CH_1^{\rm t}+\tilde{n}u\varphi_{t}\sigma_{t}
+\frac{1}{2}\tilde{n}ue^{-\phi}e^{-v}\sigma_{t}^{2}-\frac{1}{2}\tilde{n}ue^{-v}\sigma_{xt}^{2}+\tilde{n}e^{-v}\sigma_{t}\sigma_{ttx}\right]_{x=0}\\[2mm]
&+\varepsilon\beta_{1}\int_{\mathbb{R}_{+}} W_{\varepsilon-1,\beta_{1}}\left[-\CH_1^{\rm t}+\tilde{n}u\varphi_{t}\sigma_{t}
+\frac{1}{2}\tilde{n}ue^{-\phi}e^{-v}\sigma_{t}^{2}-\frac{1}{2}\tilde{n}ue^{-v}\sigma_{xt}^{2}\right]dx
\\
=&-\varepsilon\beta_{1}\int_{\mathbb{R}_{+}} W_{\varepsilon-1,\beta_{1}}\tilde{n}e^{-v}\sigma_{t}\sigma_{ttx}dx+\int_{\mathbb{R}_{+}} W_{\varepsilon,\beta_{1}}\CN_4 dx.
\end{aligned}
\end{eqnarray}

By differentiating the boundary condition \eqref{1.19} in $t$, we have
\begin{eqnarray}\label{1.19tnews}
  \sigma_{xtt}(t,0)=\left[-e^{v}\psi_{t}-u(e^{v})_{t}-u_{\infty}e^{-\sigma}\sigma_{t}\right](t,0).
\end{eqnarray}
Using the boundary condition \eqref{1.19tnews} and under the condition $\eqref{1.15b}$, we have
\begin{align}
 &\Big[-\CH_1^{\rm t}+\tilde{n}u\varphi_{t}\sigma_{t}
+\frac{1}{2}\tilde{n}ue^{-\phi}e^{-v}\sigma_{t}^{2}-\frac{1}{2}\tilde{n}ue^{-v}\sigma_{xt}^{2}+\tilde{n}e^{-v}\sigma_{t}\sigma_{ttx}\Big]_{x=0}\notag\\[2mm]
=&\Big[-\frac{\tilde{n}}{2}RTu\varphi_{t}^{2}-\tilde{n}RT\varphi_{t}\psi_{t}-\frac{\tilde{n}}{2}mu\psi_{t}^{2}-R\tilde{n}\zeta_{t}\psi_{t}-\frac{R\tilde{n}
u}{2(\gamma-1)T}\zeta_{t}^{2}\notag\\[2mm]&+\frac{1}{2}\tilde{n}ue^{-\phi}e^{-v}\sigma_{t}^{2}-\tilde{n}u_{\infty}e^{-v}e^{-\sigma}\sigma_{t}^{2}
-\frac{1}{2}\tilde{n}ue^{-v}\sigma_{xt}^{2}\Big]_{x=0}
\notag\\[2mm]
\geq& c[\varphi_{t}^{2}+\psi_{t}^{2}+\zeta_{t}^{2}+\sigma_{t}^{2}+\sigma_{xt}^{2}]_{x=0}
-C(\mathcal{N}_{\lambda,\beta}(M)+\phi_{b})[\sigma_{t}^{2}+\sigma_{xt}^{2}+\varphi_{t}^{2}+\psi_{t}^{2}+\zeta_{t}^{2}]_{x=0}
\notag\\[2mm]
\geq& (c-C\delta)[\varphi_{t}^{2}+\psi_{t}^{2}+\zeta_{t}^{2}+\sigma_{t}^{2}+\sigma_{xt}^{2}]_{x=0}.
\label{2.8cdtnew}
\end{align}

The estimate of the third term in \eqref{2.7tbnew} is similar to obtaining \eqref{2.10new}. In fact, we only need to use $(\varphi_{t},\psi_{t},\zeta_{t},\sigma_{t})$
to replace $(\varphi,\psi,\zeta,\sigma)$  in all the estimates of \eqref{2.9new} and \eqref{2.10new}. Thus, as for obtaining \eqref{2.10new}, it holds that
\begin{align}
&\varepsilon\beta_{1}\int_{\mathbb{R}_{+}} W_{\varepsilon-1,\beta_{1}}\Big[-\CH_1^{\rm t}+\tilde{n}u\varphi_{t}\sigma_{t}
+\frac{1}{2}\tilde{n}ue^{-\phi}e^{-v}\sigma_{t}^{2}-\frac{1}{2}\tilde{n}ue^{-v}\sigma_{xt}^{2}\Big]dx\notag\\[2mm]
\geq&\varepsilon\beta_{1}\int_{\mathbb{R}_{+}} W_{\varepsilon-1,\beta_{1}}\Big[\frac{RT_{\infty}}{2}|u_{\infty}|(\varphi_{t}-\frac{1}{|u_{\infty}|}\psi_{t})^{2}
+\frac{(\gamma-1)RT_{\infty}}{2|u_{\infty}|}(\psi_{t}-\frac{|u_{\infty}|}{(\gamma-1)T_{\infty}}\zeta_{t})^{2}\notag\\[2mm]&
+\frac{|u_{\infty}|}{2}(\sigma_{t}+\frac{1}{|u_{\infty}|}\psi_{t})^{2}\Big]dx-C\beta_{1}(\beta_{1}^{2}+\mathcal{N}_{\lambda,\beta}(M))
\|(\varphi_{t},\psi_{t},\zeta_{t},\sigma_{t})\|_{\varepsilon-3,\beta_{1}}^{2}\notag\\[2mm]&-C\beta_{1}\|\sigma_{xt}\|_{\varepsilon-1,\beta_{1}}^{2}
-C\beta_{1}[\sigma_{t}^{2}+\sigma_{xt}^{2}]_{x=0}\notag\\[2mm]
\geq&-C\beta_{1}(\beta_{1}^{2}+\delta\beta^{3})\|(\varphi_{t},\psi_{t},\zeta_{t},\sigma_{t})\|_{\varepsilon-3,\beta_{1}}^{2}-C\beta_{1}\|\sigma_{xt}\|_{\varepsilon-1,\beta_{1}}^{2}
-C\beta_{1}[\sigma_{t}^{2}+\sigma_{xt}^{2}]_{x=0}.
\label{2.10tgnew}
\end{align}

To estimate the first term on the right-hand side of \eqref{2.7tbnew}, we apply the Schwarz inequality and  use  Lemma \ref{main.result4} and
\eqref{1.16} as follows:
\begin{eqnarray}\label{2.12nhgfgnew}
\begin{aligned}[b]
&\left|-\varepsilon\beta_{1}\int_{\mathbb{R}_{+}} W_{\varepsilon-1,\beta_{1}}\tilde{n}e^{-v}\sigma_{t}\sigma_{ttx}dx\right|\\
\leq& C\beta_{1}\|(\sigma_{t},\sigma_{ttx})\|_{\varepsilon-1,\beta_{1}}^{2}\\
\leq& C\beta_{1}\|(\varphi_{tt},\varphi_{t})\|_{\varepsilon-1,\beta_{1}}^{2}
+ C\beta_{1}[\varphi^{2}+\psi^{2}+\sigma^{2}+\varphi_{t}^{2}+\psi_{t}^{2}+\sigma_{t}^{2}]_{x=0}\\
\leq& C\beta_{1}\|(\varphi_{x},\psi_{x},\varphi_{xt},\psi_{xt},\zeta_{x},\sigma_{x})\|_{\varepsilon-1,\beta_{1}}^{2}+C\beta_{1}\phi_{b}^{2}\int_{\mathbb{R}_{+}} W_{\varepsilon-1,\beta_{1}}G(x)^{-2}(\varphi^{2}+\psi^{2}+\zeta^{2})dx\\
&+ C\beta_{1}[\varphi^{2}+\psi^{2}+\sigma^{2}+\varphi_{t}^{2}+\psi_{t}^{2}+\sigma_{t}^{2}]_{x=0}\\
\leq& C\beta_{1}\|(\varphi_{x},\psi_{x},\varphi_{xt},\psi_{xt},\zeta_{x},\sigma_{x})\|_{\varepsilon-1,\beta_{1}}^{2}+C\beta_{1}^{5}\|(\varphi,\psi,\zeta)\|_{\varepsilon-3,\beta_{1}}^{2}\\
&+ C\beta_{1}[\varphi^{2}+\psi^{2}+\sigma^{2}+\varphi_{t}^{2}+\psi_{t}^{2}+\sigma_{t}^{2}]_{x=0}
\end{aligned}
\end{eqnarray}
Similarly for deriving  \eqref{2.13new} and \eqref{2.12nhgfgnew}, we can treat the last term of \eqref{2.7tbnew} as follows:
\begin{eqnarray}\label{2.13tgnew}
\begin{aligned}[b]
\int_{\mathbb{R}_{+}} W_{\varepsilon,\beta_{1}}\CN_4 dx  \leq&
C(\beta_{1}^{3}+\delta\beta^{3})\|(\varphi,\psi,\zeta,\varphi_{t},\psi_{t},\zeta_{t})\|_{\varepsilon-3,\beta_{1}}^{2}
\\&+C(\beta_{1}^{2}+\delta\beta^{3})\|(\psi_{x},\zeta_{x},\sigma_{x},\sigma_{xt},\varphi_{xt},\psi_{xt},\varphi_{x})\|_{\varepsilon-1,\beta_{1}}^{2}
\\&+C(\beta_{1}^{2}+\delta\beta^{3})[\sigma^{2}+\varphi^{2}+\psi^{2}+\sigma_{t}^{2}+\varphi_{t}^{2}+\psi_{t}^{2}]_{x=0},
\end{aligned}
\end{eqnarray}
where we repeatedly use Lemma \ref{main.result4} and
\eqref{1.16}.

Substituting \eqref{2.8cdtnew}-\eqref{2.13tgnew} into \eqref{2.7tbnew}, we have
\begin{align}
&\frac{d}{dt}\int_{\mathbb{R}_{+}} W_{\varepsilon,\beta_{1}}\left[\CE_1^{\rm t}+\frac{\tilde{n}}{2}e^{-\phi}e^{-v}\sigma_{t}^{2}+\frac{\tilde{n}}{2}e^{-v}\sigma_{xt}^{2}\right]
dx+c[\varphi_{t}^{2}+\psi_{t}^{2}+\zeta_{t}^{2}+\sigma_{t}^{2}+\sigma_{xt}^{2}]_{x=0}\notag\\
\leq&
C(\beta_{1}^{3}+\delta\beta^{3})\|(\varphi,\psi,\zeta,\varphi_{t},\psi_{t},\zeta_{t})\|_{\varepsilon-3,\beta}^{2}
+C(\beta_{1}+\delta\beta^{3})\|(\varphi_{x},\psi_{x},\zeta_{x},\sigma_{x},\sigma_{xt},\varphi_{xt},\psi_{xt})\|_{\varepsilon-1,\beta}^{2}\notag
\\&+C(\beta_{1}+\delta\beta^{3})[\sigma^{2}+\varphi^{2}+\psi^{2}]_{x=0},
\label{2.14klnew}
\end{align}
provided that $0<\beta_{1}\leq\beta$ and $\delta$ are sufficiently small, where $\CE_1^{\rm t}$ is defined in \eqref{1}. Furthermore, multiplying \eqref{2.14klnew} by $(1+\beta \tau)^{\xi}$ and integrating the resulting inequality
over $(0,t)$ give the desired estimate \eqref{2.uhgkyi}.  This hence completes the proof of Lemma \ref{mainboundary2g}.
\end{proof}

\begin{proof}[Proof of Proposition \ref{ste.pro1.1}]
 Now, following Lemmata $\ref{main.result2.9}$--$\ref{mainboundary2g}$  above, we are ready to
prove Proposition \ref{ste.pro1.1}. Firstly, by setting $0<\beta_{1}\ll\theta_{0}\beta\ll\beta\ll\delta\ll\theta_{0}\ll1$ roughly speaking(see \cite{boundaawa} for details), we can show  that
multiplication of \eqref{2.10} by $\theta_{0}$ and addition to \eqref{2.uykjui}, and multiplication of \eqref{2.35} by $\theta_{0}$ and addition to \eqref{2.uhgkyi}
respectively lead to the following two estimates
\begin{align}
 &(1+\beta t)^{\xi}\|(\varphi,\psi,\zeta,\sigma)(t)\|_{\varepsilon,\beta,1}^{2}\notag\\
&\quad +\int_{0}^{t} (1+\beta
 \tau)^{\xi}\left\{\beta^{3}\|(\varphi,\psi,\zeta)\|_{\varepsilon-3,\beta}^{2}
 +\beta\|(\varphi_{x},\psi_{x},\zeta_{x},\sigma_{x})\|_{\varepsilon-1,\beta}^{2}\right\}d\tau\notag\\&+\int_{0}^{t} (1+\beta
\tau)^{\xi}[\varphi^{2}+\psi^{2}+\zeta^{2}+\sigma^{2}+\sigma_{x}^{2}]_{x=0}d\tau\notag\\
 \leq &C\|(\varphi_{0},\psi_{0},\zeta_{0},\sigma_{0})\|_{\varepsilon,\beta,1}^{2}
 +C\xi\beta\int_{0}^{t} (1+\beta
 \tau)^{\xi-1}\|(\varphi,\psi,\zeta,\sigma)\|_{\varepsilon,\beta,1}^{2}
 d\tau
\label{2.10nm}
 \end{align}
and
\begin{align}
 &(1+\beta t)^{\xi}\|(\varphi_{t},\psi_{t},\zeta_{t},\sigma_{t})(t)\|_{\varepsilon,\beta,1}^{2}\notag\\
&\quad +\int_{0}^{t} (1+\beta
 \tau)^{\xi}\left\{\beta^{3}\|(\varphi_{t},\psi_{t},\zeta_{t})\|_{\varepsilon-3,\beta}^{2}
 +\beta\|(\varphi_{tx},\psi_{tx},\zeta_{tx},\sigma_{tx})\|_{\varepsilon-1,\beta}^{2}\right\}d\tau\notag\\&+\int_{0}^{t} (1+\beta
\tau)^{\xi}\left[\varphi_{t}^{2}+\psi_{t}^{2}+\zeta_{t}^{2}+\sigma_{t}^{2}+\sigma_{xt}^{2}\right]_{x=0}d\tau\notag\\
 \leq &C\|(\varphi_{0t},\psi_{0t},\zeta_{0t},\sigma_{0t})\|_{\varepsilon,\beta,1}^{2}
 +C\xi\beta\int_{0}^{t} (1+\beta
 \tau)^{\xi-1}\|(\varphi_{t},\psi_{t},\zeta_{t},\sigma_{t})\|_{\varepsilon,\beta,1}^{2}d\tau\notag\\&+C\int_{0}^{t} (1+\beta
 \tau)^{\xi}\left\{\beta^{3}\|(\varphi,\psi,\zeta,\varphi_{xx},\psi_{xx},\zeta_{xx})\|_{\varepsilon-3,\beta}^{2}
 +\beta\|(\varphi_{x},\psi_{x},\zeta_{x},\sigma_{x})\|_{\varepsilon-1,\beta}^{2}\right\}d\tau\notag\\&+C\int_{0}^{t} (1+\beta
\tau)^{\xi}[\varphi^{2}+\psi^{2}+\sigma^{2}]_{x=0}d\tau.
\label{2.35nm}
 \end{align}
Then multiply \eqref{2.35nm} by a suitably small positive constant and add to \eqref{2.10nm}. As a result, we have
\begin{align}
 &(1+\beta t)^{\xi}\|(\varphi,\psi,\zeta,\sigma,\varphi_{t},\psi_{t},\zeta_{t},\sigma_{t})(t)\|_{\varepsilon,\beta,1}^{2}\notag\\
&\quad +\int_{0}^{t} (1+\beta
 \tau)^{\xi}\left\{\beta^{3}\|(\varphi,\psi,\zeta,\varphi_{t},\psi_{t},\zeta_{t})\|_{\varepsilon-3,\beta}^{2}
 +\beta\|(\varphi_{x},\psi_{x},\zeta_{x},\sigma_{x},\varphi_{tx},\psi_{tx},\zeta_{tx},\sigma_{tx})\|_{\varepsilon-1,\beta}^{2}\right\}d\tau\notag\\
 \leq &C[\|(\varphi_{0},\psi_{0},\zeta_{0},\sigma_{0},\varphi_{0t},\psi_{0t},\zeta_{0t},\sigma_{0t})\|_{\varepsilon,\beta,1}^{2}
 +C\xi\beta\int_{0}^{t} (1+\beta
 \tau)^{\xi-1}\|(\varphi,\psi,\zeta,\sigma,\varphi_{t},\psi_{t},\zeta_{t},\sigma_{t})\|_{\varepsilon,\beta,1}^{2}d\tau\notag\\&+C\beta^{3}\int_{0}^{t} (1+\beta
 \tau)^{\xi}\|(\varphi_{xx},\psi_{xx},\zeta_{xx})\|_{\varepsilon-3,\beta}^{2}
 d\tau.
\label{2.35nmkjm}
 \end{align}
Lastly,  applying Lemma \ref{main.result4}, Lemma \ref{main.result4.3} and Lemma \ref{main.result4.4},
and taking $\delta>0$ sufficiently small, one concludes that
\begin{align}
 &(1+\beta t)^{\xi}(\|(\varphi,\psi,\zeta)(t)\|_{\varepsilon,\beta,2}^{2}+\|\sigma(t)\|_{\varepsilon,\beta,4}^{2})
\notag\\[2mm]&+\beta^{3}\int_{0}^{t}(1+\beta
\tau)^{\xi}(\|(\varphi,\psi,\zeta)(\tau)\|_{\varepsilon-3,\beta,2}^{2}+\|\sigma(\tau)\|_{\varepsilon-3,\beta,4}^{2})d\tau\notag\\[2mm]
\leq& C(\|(\varphi_{0},\psi_{0},\zeta_{0})\|_{\varepsilon,\beta,2}^{2}+r_{0}^{2})
+C\xi\beta\int_{0}^{t} (1+\beta
\tau)^{\xi-1}(\|(\varphi,\psi,\zeta)(\tau)\|_{\varepsilon,\beta,2}^{2}+\|\sigma(\tau)\|_{\varepsilon,\beta,4}^{2})d\tau.
\label{2.54}
\end{align}
In terms of \eqref{2.54}, applying an induction argument similar as \cite{SKawashima} and \cite{MNishikawa} with the choice of
  $\xi=(\lambda-\varepsilon)/3+\kappa$ for an arbitrary positive constant $\kappa$
and combining the elliptic estimates in Lemma \ref{main.result4} yield that
 \begin{align}
 &(1+\beta
 t)^{(\lambda-\varepsilon)/3+\kappa}(\|(\varphi,\psi,\zeta)(\tau)\|_{\varepsilon,\beta,2}^{2}+\|\sigma(\tau)\|_{\varepsilon,\beta,4}^{2})\notag\\[2mm]
 &
\quad+\beta^{3}\int_{0}^{t} (1+\beta
\tau)^{(\lambda-\varepsilon)/3+\kappa}
(\|(\varphi,\psi,\zeta)(\tau)\|_{\varepsilon-3,\beta,2}^{2}+\|\sigma(\tau)\|_{\varepsilon-3,\beta,4}^{2})d\tau
\notag\\[2mm]
\leq& C(1+\beta
 t)^{\kappa}(\|(\varphi_{0},\psi_{0},\zeta_{0})\|_{\lambda,\beta,2}^{2}+r_{0}^{2}),\notag
\end{align}
which proves the desired estimate \eqref{prop2.1r1} under the conditions  \eqref{b}, \eqref{c} and \eqref{d}.
Then this completes the proof of Proposition \ref{ste.pro1.1}.
\end{proof}

\section{Energy estimates for the nondegenerate case }
The aim of this section is to prove the asymptotic stability of the sheath to \eqref{1.1} under the nondegenerate condition \eqref{1.15a}. As we have seen in
the previous section concerning the degenerate problem, the essential difference
between the Dirichlet problem \cite{SIAm} and the time-evolving Neumann problem resides
in the treatment of boundary terms. This is also the case for the nondegenerate problem.
Since the derivation of a priori estimates with the Dirichlet condition for the nondegenerate problem are given in depth in \cite{SIAm}, we describe only central ideas of the proofs in this section.

\begin{proposition}\label{ste.pro2}
Let the same conditions on $T_{\infty}$, $u_{\infty}$ and
$\lambda$ as in Theorem \ref{1.2theorem} hold.

\medskip
(i) Let $(\varphi,\psi,\zeta,\sigma)$ be a solution to
\eqref{1.16}-\eqref{1.19} which satisfies
$$
(e^{\lambda
x/2}\varphi,e^{\lambda x/2}\psi,e^{\lambda x/2}\zeta,e^{\lambda
x/2}\sigma)\in (\mathscr{X}_{2}([0,M]))^{3} \times
\mathscr{X}_{2}^2([0,M])
$$
for $M>0$. Then, there exist  constants $\delta>0$ and $C>0$ independent
of $M$ such that if the following conditions
\begin{center}
$\alpha>0$, $\beta\in(0,\lambda]$, and
$\beta+(|\phi_{b}|+\mathcal {N}_{\lambda}(M)+\alpha)/\beta \leq
\delta$
\end{center}
are satisfied, where
\begin{eqnarray*}
 \begin{aligned}[b]
 \mathcal {N}_{\lambda}(M):=\sup_{0\leq t\leq
M}(\|(e^{\lambda x/2}\varphi, e^{\lambda x/2}\psi, e^{\lambda x/2}\zeta)(t)\|_{H^{2}}+|\sigma_{x}(t,0)|),
\end{aligned}
\end{eqnarray*}
then it holds for any $t\in[0,M]$ that
\begin{align}
\|(e^{\beta x/2}\varphi,e^{\beta x/2}\psi,e^{\beta x/2}\zeta)(t)\|_{H^{2}}^{2}&+\|e^{\beta
x/2}\sigma(t)\|_{H^{4}}^{2}\notag\\[2mm]
&\leq C (\|(e^{\lambda
x/2}\varphi_{0},e^{\lambda
x/2}\psi_{0},e^{\lambda
x/2}\zeta_{0})\|_{H^{2}}^{2}+r_{0}^{2})e^{-\alpha t}.\label{3.1}
\end{align}

\medskip
(ii) Let $(\varphi,\psi,\zeta,\sigma)$ be a solution to
\eqref{1.16}-\eqref{1.19} over $[0,M]$ for $M>0$. Then, for any $\varepsilon\in(0,\lambda]$,
there exist constants $\delta>0$ and $C>0$ independent of $M$ such that if all the following conditions
\begin{equation*}
((1+\beta x)^{\lambda/2}\varphi,(1+\beta x)^{\lambda/2}\psi,(1+\beta
x)^{\lambda/2}\zeta,(1+\beta x)^{\lambda/2}\sigma)\in
(\mathscr{X}_{2}([0,M]))^{3} \times
\mathscr{X}_{2}^2([0,M])
\end{equation*}
and
\begin{eqnarray}\label{3.2b}
 \begin{aligned}[b]
\beta+(|\phi_{b}|+\mathcal {N}_{\lambda,\beta}(M))/\beta \leq
\delta,\quad \beta>0
\end{aligned}
\end{eqnarray}
are satisfied, then it holds for any $t\in[0,M]$ that
\begin{eqnarray}\label{3.3}
 \begin{aligned}[b]
\|(\varphi,\psi,\zeta)(t)\|_{\varepsilon,\beta,2}^{2}+\|\sigma(t)\|_{\varepsilon,\beta,4}^{2}\leq
C (\|(\varphi_{0},\psi_{0},\zeta_{0})\|_{\lambda,\beta,2}^{2}+r_{0}^{2})(1+\beta
t)^{-(\lambda-\varepsilon)}.
\end{aligned}
\end{eqnarray}
\end{proposition}

 Since it is easier to treat the {\it a priori} estimate for the exponential weight than that for the algebraic weight, we would only prove  Proposition \ref{ste.pro2} (ii) in the case of algebraic weights for brevity.  Similarly to the degenerate problem, the proof of Proposition \ref{ste.pro2} (ii) consists in
deriving estimates Lemmata \ref{main.result2.987}, \ref{main.newloik0},  \ref{main.result2.987as},  \ref{mainnewjry2g} corresponding to Lemmata   \ref{main.result2.9}, \ref{main.resultloik0}, \ref{main.result2.erer9},  \ref{mainboundary2g}. Then,
 Proposition \ref{ste.pro2} is proved at the end of this section.

\begin{lemma}\label{main.result2.987}
Under the same conditions as in Proposition \ref{ste.pro2} (ii), for any $\varepsilon\in(0,\lambda]$, there exist constants $\delta>0$ and $C>0$ independent of $M$ such that
it holds for any $t\in[0,M]$ and $\xi\geq 0$ that
\begin{eqnarray}
&&(1+\beta t)^{\xi}\|(\varphi,\psi,\zeta)(t)\|_{\varepsilon,\beta,1}^{2}
+\beta\int_{0}^{t} (1+\beta
\tau)^{\xi} (\|(\varphi,\psi,\zeta)(\tau)\|_{\varepsilon-1,\beta,1}^{2}
+\|\sigma_{x}(\tau)\|_{\varepsilon-1,\beta}^{2})d\tau\notag\\[2mm]
&&\leq C\|(\varphi_{0},\psi_{0},\zeta_{0})\|_{\varepsilon,\beta,1}^{2}
+C\xi\beta\int_{0}^{t} (1+\beta
\tau)^{\xi-1}\|(\varphi,\psi,\zeta)(\tau)\|_{\varepsilon,\beta,1}^{2}d\tau\notag\\[2mm]
&&
+C\int_{0}^{t}(1+\beta
\tau)^{\xi}[\sigma^{2}+\sigma_{x}^{2}]_{x=0}d\tau.\label{3.9}
\end{eqnarray}
\end{lemma}

 \begin{proof}
 As in the proof of Lemma \ref{main.result2.9}, one can repeat the same procedure to obtain the identity \eqref{2.7}. It remains to re-estimate each term in \eqref{2.7}. First of all, one can still show \eqref{2.10a} in the same way so that  \eqref{2.11a} holds true. For the boundary terms on the left-hand side of \eqref{2.7}, the non-negativity estimates \eqref{2.11} and \eqref{2.11b} are also satisfied. Only the slight differences occur to estimates on $I_1$, $I_2$ and the right-hand term of $\eqref{2.7}$; it is indeed much easier to make estimates in the non-degenerate case than in the degenerate case considered before. In fact, for $I_1$, it holds that
 \begin{align}
 I_{1} &\geq\int_{\mathbb{R}_{+}}\varepsilon\beta
W_{\varepsilon-1,\beta}\Big\{
\frac{1}{2}(RT_{\infty}+1)|u_{\infty}|\varphi^{2}-RT_{\infty}\varphi\psi+\frac{1}{2}m|u_{\infty}|\psi^{2}\notag\\[2mm]
&\qquad\qquad\qquad\qquad+\frac{R
|u_{\infty}|}{2(\gamma-1)T_{\infty}}\zeta^{2}-R\zeta\psi+\sigma\psi\Big\}dx\notag\\[2mm]
&\quad
-C(\mathcal {N}_{\lambda,\beta}(M)+\phi_{b})\int_{\mathbb{R}_{+}}\varepsilon\beta
W_{\varepsilon-1,\beta}(\varphi^{2}+\psi^{2}+\zeta^{2})dx.
\label{3.12}
\end{align}
Furthermore, using the Cauchy-Schwarz inequality
$\sigma\psi\geq-(\frac{|u_{\infty}|}{2}\sigma^{2}+\frac{1}{2|u_{\infty}|}\psi^{2})$,
it follows from \eqref{3.12} that
\begin{eqnarray}\label{3.13}
 \begin{aligned}[b]I_{1} \geq&\int_{\mathbb{R}_{+}}\varepsilon\beta
W_{\varepsilon-1,\beta}\Big\{
\frac{|u_{\infty}|}{2}(RT_{\infty}+1)\varphi^{2}-RT_{\infty}\varphi\psi+\frac{m|u_{\infty}|^{2}-1}{2|u_{\infty}|}\psi^{2}\\[2mm]
&\qquad\qquad\qquad\qquad+\frac{R
|u_{\infty}|\zeta^{2}}{2(\gamma-1)T_{\infty}}-R\zeta\psi-\frac{|u_{\infty}|}{2}\sigma^{2}\Big\}dx\\[2mm] &
-C\delta\beta^{2}\int_{\mathbb{R}_{+}}
W_{\varepsilon-1,\beta}(\varphi^{2}+\psi^{2}+\zeta^{2})dx.
\end{aligned}
\end{eqnarray}
To deal with the bad term $-\int_{\mathbb{R}_{+}}\varepsilon\beta
W_{\varepsilon-1,\beta}\frac{|u_{\infty}|}{2}\sigma^{2} dx$ on the right-hand side of \eqref{3.13}, we first rewrite $\eqref{2.3}$ as the form of
\begin{eqnarray}\label{3.1phgl3}
 \begin{aligned}
\sigma_{xx}=\varphi+\sigma+(\tilde{n}-1)\varphi
+(e^{-\tilde{\phi}}-1)\sigma+\frac{\tilde{n}}{2}e^{\theta_{1}\varphi}\varphi^{2}-\frac{e^{-\tilde{\phi}}}{2}e^{-\theta_{2}\sigma}\sigma^{2},\
\ \  \theta_{1},\theta_{2}\in(0,1).
\end{aligned}
\end{eqnarray}
Then, by multiplying \eqref{3.1phgl3} by $-|u_{\infty}|\varepsilon\beta W_{\varepsilon-1,\beta}\sigma$ and using \eqref{1.14}, one has
\begin{align}
&|u_{\infty}|\varepsilon\beta\int_{\mathbb{R}_{+}}
W_{\varepsilon-1,\beta}\sigma_{x}^{2}dx+\int_{\mathbb{R}_{+}} |u_{\infty}|\varepsilon(\varepsilon-1)\beta^{2}W_{\varepsilon-2,\beta}\sigma\sigma_{x}dx
\notag\\[2mm]
&\leq -\int_{\mathbb{R}_{+}}|u_{\infty}|\varepsilon\beta W_{\varepsilon-1,\beta}\varphi \sigma dx
-\int_{\mathbb{R}_{+}}|u_{\infty}|\varepsilon\beta W_{\varepsilon-1,\beta}\sigma^{2}dx
\notag\\[2mm]
&\quad+C(\phi_{b}+\|\sigma\|_{\infty})\int_{\mathbb{R}_{+}}\varepsilon\beta
W_{\varepsilon-1,\beta}(\varphi^{2}+\sigma^{2})dx+C\beta[\sigma^{2}+\sigma_{x}^{2}]_{x=0}.
\label{3.1iup1}
\end{align}
Applying the Cauchy-Schwarz inequality $-\sigma\varphi\leq\frac{1}{2}\varphi^{2}+\frac{1}{2}\sigma^{2}$, the first two terms on the right-hand side of \eqref{3.1iup1} are bounded by
\begin{equation*}
\frac{1}{2}\int_{\mathbb{R}_{+}}|u_{\infty}|\varepsilon\beta W_{\varepsilon-1,\beta}\varphi^{2}dx
-\frac{1}{2}\int_{\mathbb{R}_{+}}|u_{\infty}|\varepsilon\beta W_{\varepsilon-1,\beta}\sigma^{2}dx.
\end{equation*}
Applying Lemma \ref{main.result4}, Lemma \ref{main.result4.2} and \eqref{3.2b}, we have
\begin{equation*}
C(\phi_{b}+\|\sigma\|_{\infty})\int_{\mathbb{R}_{+}}\varepsilon\beta
W_{\varepsilon-1,\beta}(\varphi^{2}+\sigma^{2})dx
\leq C\delta\beta\int_{\mathbb{R}_{+}}\varepsilon\beta
W_{\varepsilon-1,\beta}\varphi^{2}dx+C\delta\beta[\sigma_{x}^{2}]_{x=0}.
\end{equation*}
The second term on the left hand side of \eqref{3.1iup1} can been treated by using the integration by parts as
 \begin{eqnarray*}
 \begin{aligned}
\int_{\mathbb{R}_{+}} |u_{\infty}|\varepsilon(\varepsilon-1)\beta^{2}W_{\varepsilon-2,\beta}\sigma\sigma_{x}dx
=&-\frac{1}{2}|u_{\infty}|\varepsilon(\varepsilon-1)\beta^{2}[\sigma^{2}]_{x=0}\\&-\frac{1}{2}\int_{\mathbb{R}_{+}}|u_{\infty}|\varepsilon(\varepsilon-1)(\varepsilon-2)\beta^{3}W_{\varepsilon-3,\beta}\sigma^{2}dx,
\end{aligned}
\end{eqnarray*}
where by using Lemma \ref{main.result4}, it further holds that
\begin{equation*}
\left|\int_{\mathbb{R}_{+}}|u_{\infty}|\varepsilon(\varepsilon-1)(\varepsilon-2)\beta^{3}W_{\varepsilon-3,\beta}\sigma^{2}dx\right|\leq C\beta^3\|\varphi\|_{\varepsilon-1,\beta}^{2}+C\beta^3[\sigma_{x}^{2}]_{x=0}.
\end{equation*}
Plugging all the above estimates into \eqref{3.1iup1} gives that
 \begin{align}
&\frac{1}{2}\int_{\mathbb{R}_{+}}|u_{\infty}|\varepsilon\beta W_{\varepsilon-1,\beta}\varphi^{2}dx
-\frac{1}{2}\int_{\mathbb{R}_{+}}|u_{\infty}|\varepsilon\beta W_{\varepsilon-1,\beta}\sigma^{2}dx\notag \\[2mm]
\geq& \int_{\mathbb{R}_{+}}|u_{\infty}|\varepsilon\beta
W_{\varepsilon-1,\beta}\sigma_{x}^{2}dx
-C\delta\beta\int_{\mathbb{R}_{+}}\varepsilon\beta
W_{\varepsilon-1,\beta}\varphi^{2}dx\notag\\[2mm]
&\quad-C\beta^3\int_{\mathbb{R}_{+}}
W_{\varepsilon-1,\beta}\varphi^{2}dx-C\beta[\sigma^{2}+\sigma_{x}^{2}]_{x=0}.
\label{3.1gfiuj}
\end{align}
We then substitute \eqref{3.1gfiuj} back to \eqref{3.13} and take $\delta>0$ suitably small so as to obtain
\begin{eqnarray}\label{3.13de}
 \begin{aligned}[b]I_{1} \geq&\int_{\mathbb{R}_{+}}\varepsilon\beta
W_{\varepsilon-1,\beta}\Big\{
\frac{|u_{\infty}|}{2}RT_{\infty}\varphi^{2}-RT_{\infty}\varphi\psi+\frac{m|u_{\infty}|^{2}-1}{2|u_{\infty}|}\psi^{2}+\frac{R
|u_{\infty}|\zeta^{2}}{2(\gamma-1)T_{\infty}}-R\zeta\psi\Big\}dx\\[2mm]
&+\int_{\mathbb{R}_{+}}|u_{\infty}|\varepsilon\beta
W_{\varepsilon-1,\beta}\sigma_{x}^{2}dx
-C\delta\beta^{2}\int_{\mathbb{R}_{+}}
W_{\varepsilon-1,\beta}(\varphi^{2}+\psi^{2}+\zeta^{2})dx-C\beta[\sigma^{2}+\sigma_{x}^{2}]_{x=0}\\[2mm]
\geq&
(c-C\delta)\beta\|(\varphi,\psi,\zeta)\|_{\varepsilon-1,\beta}^{2}+c\beta\|\sigma_{x}\|_{\varepsilon-1,\beta}^{2}-C\beta[\sigma^{2}+\sigma_{x}^{2}]_{x=0},
\end{aligned}
\end{eqnarray}
where  $\frac{\gamma RT_{\infty}+1}{m}<
u^{2}_{\infty}$ is applied in the last inequality.

Now we estimate $I_{2}$ and the last term  in
$\eqref{2.7}$. In fact, it holds that
 \begin{align}
&|I_{2}|+\left|\int_{\mathbb{R}_{+}} W_{\varepsilon,\beta}\CN_{1} dx\right|\notag\\[2mm]
\leq & C(\mathcal
{N}_{\lambda,\beta}(M)+\phi_{b})\int_{\mathbb{R}_{+}}\varepsilon\beta
W_{\varepsilon-1,\beta}(\varphi^{2}+\psi^{2}+\zeta^{2}+\sigma^{2}+\varphi_{x}^{2}+\psi_{x}^{2}+\zeta_{x}^{2})dx\notag\\[2mm]
\leq&
C\beta\delta\|(\varphi,\psi,\zeta)\|_{\varepsilon-1,\beta,1}^{2}+C\beta\delta[\sigma_{x}^{2}]_{x=0},\label{3.14}
\end{align}
where we have used \eqref{1.14}, \eqref{1.16}, \eqref{3.2b},
$\lambda\geq 2$, the Cauchy-Schwarz inequality and Lemma \ref{main.result4}.

Substituting \eqref{2.11a}, \eqref{2.11}, \eqref{2.11b}, \eqref{3.13de} and \eqref{3.14}
into $\eqref{2.7}$, we have
\begin{equation}
\frac{d}{dt}\int_{\mathbb{R}_{+}} W_{\varepsilon,\beta}(e^{-\tilde{\phi}}\CE_0
+\CE_1^{\rm x}+\frac{1}{2}\tilde{n}^{2}\varphi^{2})dx
+\beta\|(\varphi,\psi,\zeta)\|_{\varepsilon-1,\beta,1}^{2}+\beta\|\sigma_{x}\|_{\varepsilon-1,\beta}^{2}
\leq C[\sigma^{2}+\sigma_{x}^{2}]_{x=0}\label{3.15}
\end{equation}
provided that $\delta>0$ is sufficiently small, where $\CE_0$ and $\CE_1^{\rm x}$ are defined in \eqref{def.e0} and \eqref{def.e1x} respectively. Therefore, the desired estimate \eqref{3.9} follows from multiplying \eqref{3.15} by $(1+\beta \tau)^{\xi}$ and integrating the resulting inequality over $(0,t)$. This then completes the proof of Lemma \ref{main.result2.987}.
\end{proof}

\begin{lemma}\label{main.newloik0}
Under the same conditions as in Proposition \ref{ste.pro2} (ii), for any $\varepsilon\in(0,\lambda]$, there exist constants $\delta>0$ and $C>0$ independent of $M$ such that
if the condition $0<\beta_{1}\leq \beta$ is satisfied, it holds for
any $t\in[0,M]$ and any $\xi\geq 0$ that
\begin{align}
&(1+\beta t)^{\xi}\|(\varphi,\psi,\zeta,\sigma,\sigma_{x})(t)\|_{\varepsilon,\beta_{1}}^{2}+c\int_{0}^{t} (1+\beta
\tau)^{\xi}[\varphi^{2}+\psi^{2}+\zeta^{2}+\sigma^{2}+\sigma_{x}^{2}]_{x=0}d\tau\notag\\
\leq& C\|(\varphi_{0},\psi_{0},\zeta_{0},\sigma_{0},\sigma_{x0})\|_{\varepsilon,\beta}^{2}
+C\xi\beta\int_{0}^{t} (1+\beta
\tau)^{\xi-1}\|(\varphi,\psi,\zeta,\sigma,\sigma_{x})(\tau)\|_{\varepsilon,\beta}^{2}d\tau\notag\\&+C\int_{0}^{t} (1+\beta
\tau)^{\xi}[(\beta_{1}+\delta\beta)\|(\varphi,\psi,\zeta)(\tau)\|_{\varepsilon-1,\beta,1}^{2}
+\beta_{1}\|\sigma_{x}(\tau)\|_{\varepsilon-1,\beta}^{2}]d\tau.\label{2.uykjuinm}
\end{align}
\end{lemma}

\begin{proof}
The proof is basically the same as that of Lemma \ref{main.resultloik0}. So, we explain the estimate of only one term which is bundled in the left-hand side of \eqref{2.7new}. Similarly to \eqref{2.10new} and under the condition \eqref{1.15a}, we have
\begin{align}
&\varepsilon\beta_{1}\int_{\mathbb{R}_{+}} W_{\varepsilon-1,\beta_{1}}\Big[-\frac{\CH_0}{\tilde{n}}+u\varphi\sigma
+\frac{1}{2}\tilde{u}e^{-\tilde{\phi}}e^{-\tilde{v}}\sigma^{2}-\frac{1}{2}e^{-\tilde{v}}\tilde{u}\sigma_{x}^{2}\Big]dx\notag\\[2mm]
\geq&\varepsilon\beta_{1}\int_{\mathbb{R}_{+}} W_{\varepsilon-1,\beta_{1}}\Big[\frac{1}{2}RT_{\infty}(-u_{\infty})\varphi^{2}-RT_{\infty}\varphi\psi+\frac{1}{2}m(-u_{\infty})\psi^{2}
-R\zeta\psi+\frac{R
(-u_{\infty})}{2(\gamma-1)T_{\infty}}\zeta^{2}\notag\\[2mm]
&+\sigma\psi+\frac{-u_{\infty}}{2}\sigma^{2}+\frac{3(-u_{\infty})}{2}\sigma_{x}^{2}\Big]dx
-C\beta_{1}(|\phi_{b}|+\mathcal {N}_{\lambda,\beta}(M))\|(\varphi,\psi,\zeta,\sigma)\|_{\varepsilon-1,\beta_{1}}^{2}\notag\\[2mm]
&-C\beta_{1}\|\sigma_{x}\|_{\varepsilon-1,\beta_{1}}^{2}
-C\beta_{1}[\sigma^{2}+\sigma_{x}^{2}]_{x=0}\notag\\[2mm]
\geq&\beta_{1}(c-C\beta\delta)\|(\varphi,\psi,\zeta,\sigma)\|_{\varepsilon-1,\beta_{1}}^{2}
-C\beta_{1}\|\sigma_{x}\|_{\varepsilon-1,\beta_{1}}^{2}
-C\beta_{1}[\sigma^{2}+\sigma_{x}^{2}]_{x=0}
\notag\\[2mm]
\geq&
-C\beta_{1}\|\sigma_{x}\|_{\varepsilon-1,\beta_{1}}^{2}
-C\beta_{1}[\sigma^{2}+\sigma_{x}^{2}]_{x=0}
\label{2.jhyunew}
\end{align}
provided that $\delta>0$ is sufficiently small.

Note that $(\varphi,\psi,\zeta,\sigma)$ appearing in  the other estimates in Lemma \ref{main.resultloik0} with the norm
$\|\cdot\|_{\varepsilon-3,\beta_{1}}$ should be replaced by $\|\cdot\|_{\varepsilon-1,\beta_{1}}$. Based on this and combined with \eqref{2.jhyunew},
we complete the proof of Lemma \ref{main.newloik0}.
\end{proof}

\begin{lemma}\label{main.result2.987as}
Under the same conditions as in Proposition \ref{ste.pro2} (ii), for any $\varepsilon\in(0,\lambda]$, there exist constants $\delta>0$ and $C>0$ independent of $M$ such that it holds for any $t\in[0,M]$ and $\xi\geq 0$ that
 \begin{align}
 &(1+\beta t)^{\xi}\|(\varphi_{t},\psi_{t},\zeta_{t})\|_{\varepsilon,\beta,1}^{2}
+\beta\int_{0}^{t} (1+\beta
\tau)^{\xi}\left(\|(\varphi_{t},\psi_{t},\zeta_{t})(\tau)\|_{\varepsilon-1,\beta,1}^{2}
+\|\sigma_{tx}(\tau)\|_{\varepsilon-1,\beta}^{2}\right)d\tau\notag\\[2mm]
&\leq C\|(\varphi_{t0},\psi_{t0},\zeta_{t0})\|_{\varepsilon,\beta,1}^{2}
+C\xi\beta\int_{0}^{t} (1+\beta
\tau)^{\xi-1}\|(\varphi_{t},\psi_{t},\zeta_{t})(\tau)\|_{\varepsilon,\beta,1}^{2}d\tau\notag\\[2mm]
&\quad+C\delta\beta\int_{0}^{t}
(1+\beta
\tau)^{\xi}\|(\varphi,\psi,\zeta)(\tau)\|_{\varepsilon-1,\beta,2}^{2}d\tau\notag\\[2mm]
&\quad+C\int_{0}^{t}(1+\beta
\tau)^{\xi}[\varphi^{2}+\psi^{2}+\sigma^{2}+\sigma_{t}^{2}]_{x=0}d\tau.\label{3.16}
\end{align}
\end{lemma}

\begin{proof}
The proof is basically the same as that of Lemma \ref{main.result2.erer9} with a similar
modification as seen in comparison between the proofs of Lemmata \ref{main.result2.9} and \ref{main.result2.987}. Note that all the estimates in Lemma \ref{main.result2.erer9} with the norm $\|\cdot\|_{\varepsilon-3,\beta}$ should be replaced by $\|\cdot\|_{\varepsilon-1,\beta}$.
Here, the details of the proof are omitted for brevity.

\end{proof}

\begin{lemma}\label{mainnewjry2g}
Under the same conditions as in Proposition \ref{ste.pro2} (ii), for any $\varepsilon\in(0,\lambda]$, there exist constants $\delta>0$ and $C>0$ independent of $M$ such that
if the condition $0<\beta_{1}\leq \beta$ is satisfied,
it holds for any $t\in[0,M]$, $\beta_{1}\leq \beta$ and $\xi\geq 0$ that
\begin{align}
&(1+\beta t)^{\xi}\|(\varphi_{t},\psi_{t},\zeta_{t},\sigma_{t},\sigma_{tx})(t)\|_{\varepsilon,\beta_{1}}^{2}+c\int_{0}^{t} (1+\beta
\tau)^{\xi}\left[\varphi_{t}^{2}+\psi_{t}^{2}+\zeta_{t}^{2}+\sigma_{t}^{2}+\sigma_{xt}^{2}\right]_{x=0}d\tau\notag\\
\leq& C\|(\varphi_{t0},\psi_{t0},\zeta_{t0},\sigma_{t0},\sigma_{tx0})\|_{\varepsilon,\beta}^{2}
+C\xi\beta\int_{0}^{t} (1+\beta
\tau)^{\xi-1}\|(\varphi_{t},\psi_{t},\zeta_{t},\sigma_{t},\sigma_{tx})(\tau)\|_{\varepsilon,\beta}^{2}d\tau\notag\\&+C(\beta_{1}+\delta\beta)\int_{0}^{t} (1+\beta
\tau)^{\xi}\|(\varphi_{xt},\psi_{xt},\zeta_{xt},\varphi_{t},\psi_{t},\zeta_{t},\varphi_{x},\psi_{x},\zeta_{x},\varphi,\psi,\zeta,\sigma_{x},\sigma_{tx})\|_{\varepsilon-1,\beta}^{2}d\tau.\label{2.uhkigkyi}
\notag\\&
+C_{4}(\beta_{1}+\delta\beta)\int_{0}^{t} (1+\beta
\tau)^{\xi}[\sigma^{2}+\varphi^{2}+\psi^{2}]_{x=0}d\tau.
\end{align}
\end{lemma}
\begin{proof}
We follow the similar steps as in deriving \eqref{2.uhgkyi} in the proof of Lemma \ref{mainboundary2g} except that the estimated way of \eqref{2.10tgnew} should be replaced by the estimated way  of \eqref{2.jhyunew}.
\end{proof}

Now, following Lemmata \ref{main.result2.987}, \ref{main.newloik0},  \ref{main.result2.987as},  \ref{mainnewjry2g} above, we are ready to give the

\begin{proof}[Proof of Proposition \ref{ste.pro2}]
Following the same argument as in the proofs of
Proposition \ref{ste.pro1.1}, we choose $0<\beta_{1}\ll\theta_{0}\beta\ll\beta\ll\delta\ll\theta_{0}\ll1$ sufficiently small. As in the proof of
Proposition \ref{ste.pro1.1}, combining a priori estimates obtained in
Lemmata $\ref{main.result2.987}$--$\ref{mainnewjry2g}$, applying Lemma \ref{main.result4}, Lemma \ref{main.result4.3}, Lemma \ref{main.result4.4}
and taking $\delta>0$ sufficiently small, we conclude that
\begin{align}
 &(1+\beta t)^{\xi}(\|(\varphi,\psi,\zeta)(t)\|_{\varepsilon,\beta,2}^{2}+\|\sigma(t)\|_{\varepsilon,\beta,4}^{2})\notag
\\&+\beta\int_{0}^{t} (1+\beta
\tau)^{\xi}(\|(\varphi,\psi,\zeta)(\tau)\|_{\varepsilon-1,\beta,2}^{2}+\|\sigma(\tau)\|_{\varepsilon-1,\beta,4}^{2})d\tau
\notag \\[2mm]
&\leq
 C(\|(\varphi_{0},\psi_{0},\zeta_{0})\|_{\varepsilon,\beta,2}^{2}+r_{0}^{2})
+C\xi\beta\int_{0}^{t} (1+\beta
\tau)^{\xi-1}(\|(\varphi,\psi,\zeta)(\tau)\|_{\varepsilon,\beta,2}^{2}+\|\sigma(\tau)\|_{\varepsilon,\beta,4}^{2})d\tau
\label{3.17}
\end{align}
for any $t\geq 0$ and $\xi\geq 0$.
Then, in terms of \eqref{3.17}, employing the induction argument similar as \cite{SKawashima} and \cite{MNishikawa} with $\xi=\lambda-\varepsilon+\kappa$  for an arbitrary positive constant $\kappa$
and combining the elliptic estimates in Lemma \ref{main.result4} yield that for any $t\geq 0$,
 \begin{align*}
 &(1+\beta
 t)^{\lambda-\varepsilon+\kappa}(\|(\varphi,\psi,\zeta)(\tau)\|_{\varepsilon,\beta,2}^{2}+\|\sigma(\tau)\|_{\varepsilon,\beta,4}^{2})\notag \\[2mm]
 &
\quad+\beta\int_{0}^{t} (1+\beta
\tau)^{\lambda-\varepsilon+\kappa}(\|(\varphi,\psi,\zeta)(\tau)\|_{\varepsilon-1,\beta,2}^{2}+\|\sigma(\tau)\|_{\varepsilon-1,\beta,4}^{2})d\tau
\notag \\[2mm]
&\leq C(1+\beta
 t)^{\kappa}(\|(\varphi_{0},\psi_{0},\zeta_{0})\|_{\lambda,\beta,2}^{2}+r_{0}^{2}),
\end{align*}
which proves \eqref{3.3}. This then completes the proof of the second part (ii)  of Proposition \ref{ste.pro2}. As mentioned before, for the part (i) corresponding to the exponential weight case, the proof of \eqref{3.1} follows in a similar way and thus is omitted for brevity. We therefore conclude the proof of Proposition \ref{ste.pro2}.
\end{proof}

\section{Appendix}\label{Append}
In this appendix, we will give some basic results used in the proof of Proposition \ref{ste.pro1.1} and Proposition \ref{ste.pro2}. Those lemmas below are similar to ones obtained in \cite{SIAm} and \cite{boundaawa}.

\begin{lemma}[see \cite{SIAm,boundaawa}]\label{main.result4}
Consider the elliptic equation \eqref{1.17}. Under the same assumptions as in either Proposition \ref{ste.pro1.1} for the
degenerate case or Proposition \ref{ste.pro2} (ii) for the nondegenerate case, the following estimates hold for certain positive constants $c$ and $C$
provided that $\delta$ is sufficiently small:
\begin{eqnarray*}\label{4.1lod}
 \begin{aligned}[b]
\|\sigma\|_{\alpha,\beta}^{2}+c\|\sigma_{x}\|_{\alpha,\beta}^{2}\leq
 C\|\varphi\|_{\alpha,\beta}^{2}+C[\sigma_{x}^{2}]_{x=0},
\end{aligned}
\end{eqnarray*}
\begin{eqnarray*}\label{4.1d}
 \begin{aligned}[b]
 [\sigma^{2}]_{x=0}\leq
 C\|\varphi\|_{\alpha,\beta}^{2}+C[\sigma_{x}^{2}]_{x=0},
\end{aligned}
\end{eqnarray*}
\begin{eqnarray*}\label{4.1g}
 \begin{aligned}[b]
 [\sigma_{x}^{2}]_{x=0}\leq
 C\|\varphi\|_{\alpha,\beta}^{2}+C[\sigma^{2}]_{x=0},
\end{aligned}
\end{eqnarray*}
\begin{eqnarray*}\label{4.1jh}
 \begin{aligned}[b]
 \|\sigma_{t}\|_{\alpha,\beta,1}^{2}\leq
 C\|\varphi_{t}\|_{\alpha,\beta}^{2}+C[\varphi^{2}+\psi^{2}+\sigma^{2}]_{x=0},
\end{aligned}
\end{eqnarray*}
\begin{eqnarray*}\label{4.1mk}
 \begin{aligned}[b]
\frac{d}{dt}[e^{-\sigma}-1+\sigma]_{x=0}+c\|\sigma_{t}\|_{\alpha,\beta,1}^{2}\leq
 C\|\varphi_{t}\|_{\alpha,\beta}^{2}+C[\varphi^{2}+\psi^{2}]_{x=0}\leq
 C\|(\varphi,\psi)\|_{\alpha,\beta,1}^{2},
\end{aligned}
\end{eqnarray*}
\begin{eqnarray*}\label{4.1jh}
 \begin{aligned}[b]
 \|\sigma_{tt}\|_{\alpha,\beta,1}^{2}\leq
 C\|(\varphi_{tt},\varphi_{t},\sigma_{t})\|_{\alpha,\beta}^{2}+C[\varphi_{t}^{2}+\psi_{t}^{2}+\sigma_{t}^{2}]_{x=0},
\end{aligned}
\end{eqnarray*}
\begin{eqnarray*}\label{4.1lop}
 \begin{aligned}[b]
 \|\partial_{t}^{i}\sigma\|_{\alpha,\beta,j}^{2}\leq
 C\|\varphi\|_{\alpha,\beta,i+j-2}^{2}+C[\sigma_{x}^{2}]_{x=0},\  i\in
 \mathbb{Z}\cap[0,2], \ \ \ j\in\mathbb{Z}\cap[2,4-i].
\end{aligned}
\end{eqnarray*}
\end{lemma}

\begin{lemma}[see \cite{SIAm,boundaawa}]\label{main.result4.2}
Under the same assumptions as in either Proposition \ref{ste.pro1.1} for the
degenerate case or Proposition \ref{ste.pro2} (ii) for the nondegenerate case,
it holds for any $t\in[0,M]$ and $\alpha \leq\lambda/2$ that
\begin{eqnarray*}
&\dis \|((1+\beta x)^{\alpha}{(\varphi,\psi,\zeta)},(1+\beta x)^{\alpha}{(\varphi_{x},\psi_{x},\zeta_{x})})(t)\|_{L^{\infty}(\mathbb{R}^{+})}\leq
 C\mathcal {N}_{\lambda,\beta}(M),\label{4.3}\\[2mm]
&\dis \|(1+\beta x)^{\alpha}{(\varphi_{t},\psi_{t},\zeta_{t})}(t)\|_{L^{\infty}(\mathbb{R}^{+})}\leq
 C\mathcal {N}_{\lambda,\beta}(M),\label{4.4}\\[2mm]
&\dis \|(1+\beta x)^{\alpha}{(\sigma,\sigma_{x},\sigma_{t})}(t)\|_{L^{\infty}(\mathbb{R}^{+})}\leq
 C\mathcal {N}_{\lambda,\beta}(M).\label{4.vb4}
\end{eqnarray*}
\end{lemma}

\begin{lemma}[see \cite{SIAm,boundaawa}]\label{main.result4.3}
Under the same assumptions as in either Proposition \ref{ste.pro1.1} for the
degenerate case or Proposition \ref{ste.pro2} (ii) for the nondegenerate case, the following estimates hold for certain positive constants $c$ and $C$
provided that $\delta$ is sufficiently small:
\begin{eqnarray*}
&\dis \|{(\varphi_{t},\psi_{t},\zeta_{t})}\|_{\xi,\beta}^{2}
 \leq C\|{(\varphi,\psi,\zeta)}\|_{\xi,\beta,1}^{2}+C[\sigma_{x}^{2}]_{x=0},\label{4.6}\\[2mm]
&\dis \|({\varphi_{tx},\psi_{tx},\zeta_{tx}},{\varphi_{tt},\psi_{tt},\zeta_{tt}})\|_{\xi,\beta}^{2}
 \leq C\|{(\varphi,\psi,\zeta)}\|_{\xi,\beta,2}^{2}+C[\sigma_{x}^{2}]_{x=0},\label{4.7}
\end{eqnarray*}
\begin{eqnarray*}\label{4.5}
 \begin{aligned}[b]
 \|\partial_{t}^{i}{(\varphi,\psi,\zeta)}\|_{\xi,\beta,j}^{2}\leq
 C\|{(\varphi,\psi,\zeta)}\|_{\xi,\beta,i+j}^{2}+C[\sigma_{x}^{2}]_{x=0},
\end{aligned}
\end{eqnarray*}
where $(i,j)\in\{(i,j)\in\mathbb{Z}^{2}|i,j\geq 0, i+j\leq 2\}$.
\end{lemma}

\begin{lemma}[see \cite{SIAm,boundaawa}]\label{main.result4.4}
Under the same assumptions as in either Proposition \ref{ste.pro1.1} for the
degenerate case or Proposition \ref{ste.pro2} (ii) for the nondegenerate case, the following estimates hold for certain positive constants  $C$
provided that $\delta$ is sufficiently small:
\begin{eqnarray*}
&\dis \|{(\varphi_{x},\psi_{x},\zeta_{x})}\|_{\xi,\beta}^{2}\leq
 C\|({\varphi_{t},\psi_{t},\zeta_{t}},{\varphi,\psi,\zeta})\|_{\xi,\beta}^{2}+C[\sigma_{x}^{2}]_{x=0},\\[2mm]
&\dis \|{(\varphi_{xx},\psi_{xx},\zeta_{xx})}\|_{\xi,\beta}^{2}\leq
 C\|({\varphi_{xt},\psi_{xt},\zeta_{xt}},{\varphi_{x},\psi_{x},\zeta_{x}},{\varphi,\psi,\zeta})\|_{\xi,\beta}^{2}+C[\sigma_{x}^{2}]_{x=0}.
\end{eqnarray*}
\end{lemma}

\section*{Acknowledgement}
The authors were supported by the National Natural Science Foundation of China $\#$12071163, $\#$12171390, and the Natural Science Foundation of Fujian Province of China $\#$2021J01305, $\#$2020J01071. The authors would like to thank Professor Lei Yao for many fruitful
discussions on the topic of this paper.

\end{document}